\documentclass[a4paper]{amsart}

\usepackage{geometry}[a4paper,top=1.5cm,bottom=1.5cm,left=2cm,right=2cm]
\usepackage{bbm}
\usepackage{amsmath}
\usepackage{amsfonts}
\usepackage{amssymb}
\usepackage{amsthm}
\usepackage{enumerate}
\usepackage{color}
\usepackage{mathrsfs}
\usepackage{stmaryrd}
\usepackage{hyperref}

\SetSymbolFont{stmry}{bold}{U}{stmry}{m}{n}

\newtheorem{thm}{Theorem}

\newtheorem{lem}[thm]{Lemma}
\newtheorem{prop}[thm]{Proposition}

\theoremstyle{definition}
\newtheorem{defn}[thm]{Definition}
\newtheorem*{defn*}{Definition}

\theoremstyle{remark}
\newtheorem{rem}[thm]{Remark}

\newcommand{\deq}{\mathrel{\mathop:}=}
\newcommand{\e}[1]{\mathrm{e}^{#1}}
\newcommand{\R} {\mathbb{R}}
\newcommand{\C} {\mathbb{C}}
\newcommand{\N} {\mathbb{N}}
\newcommand{\Z} {\mathbb{Z}}

\newcommand{\E} {\mathbb{E}}

\newcommand{\adj}{^{*}} 

\newcommand{\mr}[1]{\mathring{#1}}

\newcommand{\brkt}[1]{\langle #1 \rangle}

\newcommand{\bbrktt}[1]{\llbracket #1 \rrbracket}

\DeclareMathOperator{\diag}{diag}

\DeclareMathOperator{\Tr}{Tr}

\DeclareMathOperator{\re}{\mathrm{Re}}
\DeclareMathOperator{\im}{\mathrm{Im}}

\newcommand{\caB}{{\mathcal B}}
\newcommand{\caD}{{\mathcal D}}
\newcommand{\caH}{{\mathcal H}}

\newcommand{\caS}{{\mathcal S}}
\newcommand{\caU}{{\mathcal U}}
\newcommand{\caX}{{\mathcal X}}
\newcommand{\caY}{{\mathcal Y}}

\newcommand{\bbD}{{\mathbb D}}
\newcommand{\bbE}{{\mathbb E}}
\newcommand{\bbI}{{\mathbb I}}

\newcommand{\frX}{{\mathfrak X}}
\newcommand{\frY}{{\mathfrak Y}}

\newcommand{\bsk}{{\boldsymbol k}}
\newcommand{\bsl}{{\boldsymbol l}}
\newcommand{\bsu}{{\boldsymbol u}}
\newcommand{\bsx}{{\boldsymbol x}}
\newcommand{\bsy}{{\boldsymbol y}}
\newcommand{\bsA}{{\boldsymbol A}}

\newcommand{\rmm}{\mathrm{m}}

\newcommand{\wt}{\widetilde}
\newcommand{\ol}{\overline}
\newcommand{\ul}{\underline}
\newcommand{\wh}{\widehat}
\newcommand{\beq}{ \begin{equation} }
	\newcommand{\eeq}{ \end{equation} }
\newcommand{\beqs}{\begin{equation*}}
	\newcommand{\eeqs}{\end{equation*}}

\newcommand{\lone}{\mathbbm{1}} 

\newcommand{\dd}{\mathrm{d}}
\newcommand{\ii}{\mathrm{i}}

\newcommand{\AND}{\quad\text{and}\quad}

\newcommand\norm[1]{\Vert#1\Vert}
\newcommand\Norm[1]{\left\Vert #1 \right\Vert}

\newcommand\expct[1]{\mathbb{E}[#1]}
\newcommand\Expct[1]{\mathbb{E}\left[#1\right]}

\newcommand\prob[1]{\mathbf{P}\left[#1\right]}

\newcommand\Absv[1]{\left\vert#1\right\vert}
\newcommand\absv[1]{\vert#1\vert}

\newcommand\ctr[1]{\bbI[#1]}

\numberwithin{equation}{section} 
\numberwithin{thm}{section}

\title{Functional CLT for Non-Hermitian Random Matrices}
\author{L\'{a}szl\'{o} Erd\H{o}s$^{\dagger}$}
\author{Hong Chang Ji$^{\ddagger}$}
\address{IST Austria, Am Campus 1, 3400 Klosterneuburg, Austria}
\email{lerdos@ist.ac.at}
\email{hongchangji@ist.ac.at}
\thanks{$^\dagger$ Partially supported by ERC Advanced Grant "RMTBeyond" No. 101020331}
\thanks{$\ddagger$ Supported by ERC Advanced Grant "RMTBeyond" No. 101020331}
\date{\today}
\subjclass[2010]{60B20,15B52} 
\keywords{Linear statistics, central limit theorem, Gaussian free field}

\begin{document}
	
	\maketitle
	
	\begin{abstract}
		For large dimensional non-Hermitian random matrices $X$ with real or complex independent, identically distributed, centered entries, we consider the fluctuations of $f(X)$ as a matrix where $f$ is an analytic function around the spectrum of $X$. We prove that for a generic bounded square matrix $A$, the quantity $\Tr f(X)A$ exhibits Gaussian fluctuations as the matrix size grows to infinity, which consists of two independent modes corresponding to the tracial and traceless parts of $A$. We find a new formula for the variance of the traceless part that involves the Frobenius norm of $A$ and the $L^{2}$-norm of $f$ on the boundary of the limiting spectrum.
	\end{abstract}
	
	\section{Introduction}
	
	A distinctive feature of eigenvalues of large dimensional random matrices as point processes is the strong correlation between them. In particular for a Hermitian random matrix, this feature reflects the original idea of Wigner that random matrices can serve as universal models for strongly correlated system. One empirical evidence, among many others, of the strong correlation is the size of the variances of their linear statistics. Consider the eigenvalues $\{\sigma_{i}\}_{1\leq i\leq N}$ of an $(N\times N)$ random matrix $X$ with independent, identically distributed (i.i.d.) entries with zero mean and variance $1/N$. Then the celebrated circular law \cite{Bai1997, Ginibre1965, Girko1984, Tao-Vu2010} states that the distribution of a randomly chosen eigenvalue $\sigma_{i}$, or equivalently the empirical spectral distribution (e.s.d.) of $X$, weakly converges to the uniform distribution on the unit disk $\bbD$. In this setting, the linear statistics $\Tr f(X)=\sum_{i=1}^{N}f(\sigma_{i})$ for a regular test function $f$ has fluctuations of constant order, which happens to be Gaussian \cite{Rider-Silverstein2006}. Compared to the classical central limit theorem where $\sigma_{i}$'s are independent random variables, the linear statistics of $X$ does not need to be scaled by $N^{-1/2}$ showing that its eigenvalues have strong correlation.
	
	The question we ask in this paper is whether a \emph{functional} version of the central limit theorem holds for $X$. More precisely, for an i.i.d. random matrix $X$ above and a deterministic norm-bounded $(N\times N)$ matrix $A$, we are interested in the fluctuations of the functional statistics 
	\beq\label{eq:Cauchy}
	\Tr f(X)A\deq\frac{1}{2\pi\ii}\oint_{\gamma}f(z)\Tr ((z-X)^{-1}A)\dd z,
	\eeq
	where $f$ is analytic in a neighborhood $\caD$ of the closed unit disk $\ol{\bbD}$ and $\gamma$ is a positively oriented curve in $\caD$ encircling the spectrum $\sigma(X)$ of $X$. Note that this quantity is well-defined only when $\sigma(X)\subset\caD$, but we will see below that such an event has very high probability. Our main result, Theorem \ref{thm:main}, proves that $\Tr f(X)A$ has Gaussian fluctuations of constant order.
	
	Our paper is motivated by the analogous question for Wigner matrices answered in \cite{Cipolloni-Erdos-Schroder2020arXiv, Lytova2013}. For a Wigner matrix $W$, a Hermitian random matrix with independent, centered entries of variance $1/N$, it is well known that its e.s.d. converges to the semi-circle distribution in $[-2,2]$ from \cite{Wigner1955} and that the linear statistics $\Tr f(W)$ has Gaussian fluctuations  (see \cite{Lytova-Pastur2009a} for example). Furthermore, it was proved in \cite{Cipolloni-Erdos-Schroder2020arXiv} that the functional statistics $\Tr f(W)A$ has Gaussian fluctuations in the limit $N\to\infty$, where the test function can be chosen in both macro- and mesoscopic regimes, that is, $f(X)\deq g(N^{\alpha}(x-E))$ for some $\alpha\in[0,1)$ and $E\in (-2,2)$. The motivation to study functional statistics is more evident in this setting, since it relates to overlaps of eigenstates via spectral decomposition
	\beqs
		\Tr f(W)A=\sum_{i=1}^{N}f(\lambda_{i})\brkt{\bsu_{i},A\bsu_{i}},\qquad W\bsu_{i}=\lambda\bsu_{i},\,\,\norm{\bsu_{i}}=1.
	\eeqs
	
	In our setting, since $X$ is not normal, the matrix-function $f(X)$ itself is defined only when $f$ is analytic in a domain $\caD$ containing the spectrum of $X$. In particular all eigenvalues of $X$ must contribute to $f(X)$ and we cannot go below the macroscopic regime when considering $f(X)$. This differentiates $\Tr f(X)A$ from the \emph{tracial} statistics $\Tr f(X)$, for which we can use the fact that
	\beq\label{eq:spec}
		\Tr f(X)=\sum_{i=1}^{N}f(\sigma_{i})
	\eeq
	to reduce the problem to analyzing the e.s.d. of $X$. The identity \eqref{eq:spec} enables us to define $\Tr f(X)$ for more general test functions $f$, and it was proved in \cite{Cipolloni-Erdos-Schroder2019arXiv} that $\Tr f(X)$ has Gaussian fluctuations when $f$ is $H^{2+}$ or on a mesoscopic scale starting from this.
	Interestingly, the variance of $\Tr f(X)A$ we found in Theorem \ref{thm:main} involves the norm of $L^{2}$-Hardy space, which does not seem to have canonical generalization to non-analytic functions, especially to those vanishing on the boundary $\partial\bbD$.
	
	Another obstacle in studying the functional statistics $\Tr f(X)A$ is the absence of Girko's formula, which is a commonly used technique to study the eigenvalues of non-Hermitian random matrices. For example, Girko's formula played a vital role in \cite{Alt-Erdos-Kruger2018,Alt-Erdos-Kruger2021, Cipolloni-Erdos-Schroder2019arXiv}, and even in the original proofs of the circular law in \cite{Bai1997,Girko1984}. In terms of $X$, a version of Girko's formula in \cite{Tao-Vu2010} is as follows; for a smooth and compactly supported complex function $f$, we have
	\beqs
		\sum_{i=1}^{N}f(\sigma_{i})=-\frac{1}{4\pi}\int_{\C}\Delta f(z)\int_{0}^{\infty}\im \Tr (W_{z}-\ii\eta)^{-1}\dd \eta\dd^{2} z, \qquad
		W_{z}\deq\begin{pmatrix}
			0 & X-z \\ X\adj -\ol{z} & 0
		\end{pmatrix}.
	\eeqs
	Using this formula, we can focus on the Green function of the Hermitian matrix $W_{z}$ along the imaginary axis, not that of $X$, so that we can apply robust approaches developed for Hermitian random matrices such as local laws. The most crucial drawback of Girko's formula is that it applies only to tracial quantities involving the eigenvalues, not the matrix itself. In the same vein as the previous paragraph, the functional statistics $\Tr f(X)A$ concerns both eigenvalues and eigenvectors, thus we need a different approach.
	
	Nevertheless, we find that one of the four $(N\times N)$ blocks of $(W_{z}-\ii\eta)^{-1}$ can serve as an approximation to $(X-z)^{-1}$ when $\eta$ is small enough. Applying this observation to \eqref{eq:Cauchy} we may work with the contour integral of the block of $(W_{z}-\ii\eta)^{-1}$, which effectively replaces Girko's formula. We remark that a similar argument was used in \cite{Erdos-Kruger-Renfrew2018} to express $f(X)g(X)\adj$ as the contour integral of certain block of a Hermitian resolvent. Our argument requires $\eta$ to be very small, well below the optimal scale $\eta\gg 1/N$ typically appearing in local laws for $W_{z}$ (see \cite{Alt-Erdos-Kruger2018} for instance). While it is hard to keep track of $(W_{z}-\ii\eta)^{-1}$ for $\eta\ll 1/N$ when $z$ is inside the limiting spectrum $\bbD$, since we are assuming $\ol{\bbD}\subset\caD$, we can take the contour $\gamma$, hence $z$, in $\caD\setminus\ol{\bbD}$. This allows us to apply the optimal local law for $W_{z}$ outside the spectrum with all ranges of $\eta$ established e.g. in \cite{Alt-Erdos-Kruger-Nemish2019}. 
	
	After approximating $\Tr f(X)A$ with certain functional statistics of $(W_{z}-\ii\eta)^{-1}$, we prove that its moments asymptotically satisfy a recursion which leads to convergence of moments via Wick's theorem. Along the calculation, we find that quantities corresponding to the variances are normalized traces of product of three resolvents with deterministic matrices in between. We adapt the strategy of \cite{Cipolloni-Erdos-Schroder2019arXiv,Cipolloni-Erdos-Schroder2020arXiv} to derive their deterministic approximates, using local laws for product of two resolvents proved in \cite{Cipolloni-Erdos-Schroder2019arXiv}.
	
	We conclude this section with a brief history of Gaussian fluctuations in linear statistics of random matrices. For Hermitian random matrices, tracial CLT has been a classical topic and there have been many results with increasing generality in terms of regularity of $f$; see  \cite{Bai-Yao2005,He-Knowles2017,Lytova-Pastur2009a,Sosoe-Wong2013} for instance. As mentioned above, a functional CLT for Wigner matrices was proved in \cite{Cipolloni-Erdos-Schroder2020arXiv,Lytova2013}, the former covering also mesoscopic regimes. For non-Hermitian matrices, \cite{Coston-ORourke2020,Nourdin-Peccati2010,Rider-Silverstein2006} presented tracial CLT on macroscopic scale for analytic test functions, and later in \cite{Cipolloni-Erdos-Schroder2019arXiv} it was generalized to $H^{2}$ test functions both in macroscopic and partially mesoscopic scales.
	
	The rest of this paper is organized as follows: In Section \ref{sec:model}, we rigorously define our model and state our main result, Theorem \ref{thm:main}. In Section \ref{sec:outline}, we explain the key components of our proof and prove our main result based on them. The CLT for the resolvents of $W_{z}$  for complex-valued $X$ is presented in Section \ref{sec:prf_CLT_G}, which is the main technical
	achievement of our paper. Section \ref{sec:prf_tech} collects computational and technical parts of the proof of the resolvent CLT. Finally, in Section \ref{sec:real} we prove the main result for real-valued $X$.
	
	\subsection{Notational conventions}
		We introduce some conventions used throughout the paper. For each $r>0$, $\bbD_{r}$ denotes the open, centered disk in $\C$ with radius $r$ and $\bbD_{r}^{c}$ stands for its complement in $\C$. For integers $m$ and $n$, we write $\bbrktt{m,n}\deq [m,n]\cap \Z$. When $m=1$, we further abbreviate $\bbrktt{n}\equiv\bbrktt{1,n}$.  We denote the identity matrix of any dimension by $I$. For an $(N\times N)$ matrix $A$, we write $\brkt{A}\deq N^{-1}\Tr A$ for its normalized trace.
		For a $2N\times 2N$ matrix $P$, we write $\{P^{[kl]}\}_{k,l\in\bbrktt{2}}$ to denote its $(N\times N)$ blocks, and conversely for an $(N\times N)$ matrix $A$, we denote by $A^{\{kl\}}$ the $(2N\times 2N)$ matrix satisfying $(A^{\{kl\}})^{[k'l']}=A\delta_{kk'}\delta_{ll'}$; for example,
		\beqs
			P=\begin{pmatrix}
				P^{[11]} & P^{[12]}\\
				P^{[21]} & P^{[22]}
			\end{pmatrix},\qquad
			A^{\{12\}}=\begin{pmatrix}
				0 & A \\ 0 & 0
			\end{pmatrix}.
		\eeqs
		We will often use that $\Tr P^{[ij]}A=\Tr PA^{\{ji\}}$. We use the following abbreviations for sums over $\bbrktt{N}$;
		\begin{align*}
			&\sum_{a_{1},\cdots,a_{n}}\equiv \prod_{i=1}^{n}\sum_{a_{i}=1}^{N},	& &\sum_{a_{1},\cdots,a_{n}}^{(b_{1},\cdots,b_{m})}=\prod_{i=1}^{n}\sum_{\substack{a_{i}=1 \\ a_{i}\neq b_{1},\ldots,b_{m}}}^{N},\quad b_{1},\ldots,b_{m}\in\bbrktt{N}.
		\end{align*}
	
		For a random variable, vector, or matrix $X$, we write $\ctr{X}=X-\expct{X}$. 
		Given two $N$-dependent random vectors $Y$ and $Z$ of the same dimension $n$ with all moments finite and a function $f:\N\to[0,\infty)$, we write
		\beqs
		Y\overset{m}{=}Z+O_{\mathrm{m}}(f(N))
		\eeqs
		if the following holds; for each fixed monic $*$-monomial $p$ of degree $k$, there is a constant $C_{n,k}>0$ such that
		\beqs
		\absv{\expct{p(Y)}-\expct{p(Z)}}\leq C_{n,k}f(N).
		\eeqs
		In other words, $Y$ and $Z$ are $f$-close to each other in the sense of all moments. Note that this concept is really meaningful when $Y$ and $Z$ are of order $1$ and $f(N)\ll 1$, and we apply it only in such occasions.
		
		For two $N$-dependent random variables $X$ and $Y$ with $Y\geq0$, we write $X\prec Y$ or $X=O_{\prec}(Y)$ if the following holds; for all fixed $\epsilon,D>0$, there exists $N_{0}(\epsilon,D)\in\N$ such that
		\beqs
			\prob{\absv{X}>N^{\epsilon}Y}\leq N^{-D},\quad \forall N\geq N_{0}(\epsilon,D).
		\eeqs
		In this case we say that $X$ is \emph{stochastically dominated} by $Y$.
		
		Finally, we use the standard big-$O$ and $\lesssim$ notations; for two functions $f$ and $g$ of $N$ with $g\geq 0$, we write $f\lesssim g$ or $f=O(g)$ when there is a constant $C>0$ such that $f(N)\leq Cg(N)$ for all $N$.

\section{Model and results}\label{sec:model}
	\begin{defn}\label{defn:X}
		Let $\chi$ be either  real or genuinely complex random variable that satisfies the following conditions;
		\begin{itemize}
			\item $\expct{\chi}=0$ and $\expct{\absv{\chi}^{2}}=1$;
			\item for each $p\in\N$, there exists $C_{p}>0$ such that $\expct{\absv{\chi}^{p}}\leq C_{p}$;
			\item $\expct{\chi^{2}}=0$ if $\chi$ is complex-valued\footnote{We made this commonly used assumption only 
			for the sake of simplicity; our method
			can easily handle the exact dependence on arbitrary $\expct{\chi^{2}}\in \C$ as well.}.
		\end{itemize}
		We define $X$ to be the $(N\times N)$ random matrix whose entries are independent and identically distributed (i.i.d.) with the same law as $N^{-1/2}\chi$.
	\end{defn}

	As mentioned in the introduction, a proper definition of $f(X)$ must be restricted to an event of the form $\sigma(X)\subset\caD$. We introduce our choice of such an event in the next lemma, which also shows that this event has high probability.
	\begin{lem}\cite[Lemma 6.1]{Alt-Erdos-Kruger-Nemish2019}\label{lem:sing}
		For each $\delta,\kappa>0$, define the $N$-dependent event $\Omega_{\delta}(\kappa)$ as follows;
		\beq\label{eq:defn_Omega}
		\Omega_{\delta}(\kappa)\deq \left\{\sup_{z\in \bbD_{1+\delta/2}^{c}}\norm{((X-z)(X-z)\adj)^{-1}}\leq \kappa^{-2}\right\}.
		\eeq
		For any fixed $\delta>0$ and $\kappa\in(0,\delta/2)$, we have $\prob{\Omega_{\delta}(\kappa)}>1-N^{-D}$ for all $D>0$ when $N$ is sufficiently large.
	\end{lem}
	Note that the definition \eqref{eq:defn_Omega} of $\Omega_{\delta}(\kappa)$ already
	 implies $\sigma(X)\subset \bbD_{1+\delta/2}$ on the event. 
	Next, we introduce our main result, the central limit theorem for $\Tr f(X)A$ restricted on $\Omega_{\delta}(\kappa)$;
	\begin{thm}\label{thm:main}
		Let $n\in\N$, $\delta>0$ and $\kappa\in(0,\delta/2)$ be fixed, $A_{1},\cdots,A_{n}$ be $(N\times N)$ deterministic matrices with $\norm{A_{i}}\leq 1$, and $f_{1},\cdots,f_{n}$ be fixed analytic functions on $\bbD_{1+\delta}$ with $\norm{f_{i}}_{L^{\infty}(\bbD_{1+\delta/2})}\leq 1$. Define
		\beq\label{eq:defn_L}
			L_{N}(f_{i},A_{i})\deq\lone_{\Omega_{\delta}(\kappa)}\Tr f_{i}(X)A_{i},\qquad i\in\bbrktt{n}.
		\eeq
		Then the centered random vector $(\ctr{L_{N}(f_{i},A_{i})})_{i=1,\cdots,n}$ is approximately Gaussian, that is, for all $\epsilon>0$ we have
		\beq\label{eq:main}
			(\ctr{L_{N}(f_{i},A_{i})})_{i=1,\cdots,n}\overset{\mathrm{m}}{=}(\xi(f_{i},A_{i}))_{i=1,\cdots,n}+O_{\mathrm{m}}(N^{-1/2+\epsilon}),
		\eeq
		where $(\xi(f_{i},A_{i}))_{i=1,\cdots,n}$ is a centered complex Gaussian vector whose covariance is given by
		\beq\label{eq:cov}
		\begin{split}
			\expct{\xi(f_{i},A_{i})\ol{\xi(f_{j},A_{j})}}&=\frac{1}{\pi}\brkt{A_{i}}\ol{\brkt{A_{j}}}\brkt{f_{i}',f_{j}'}_{L^{2}(\bbD)}+\brkt{\mr{A}_{i}\mr{A}_{j}\adj}\brkt{f_{i},f_{j}}_{\wh{L}^{2}(\partial\bbD)},	\\
			\expct{\xi(f_{i},A_{i})\xi(f_{j},A_{j})}&=
			\begin{cases}
				0 & \text{if $\chi$ is complex,}\\
				\frac{1}{\pi}\brkt{A_{i}}\brkt{A_{j}}\brkt{f_{i}',(f_{j}\adj)'}_{L^{2}(\bbD)}+\brkt{\mr{A}_{i}\mr{A}_{j}^{\intercal}}\brkt{f_{i},f_{j}\adj}_{\wh{L}^{2}(\partial\bbD)}& \text{if $\chi$ is real}.
			\end{cases}
		\end{split}
		\eeq
		Here we denoted $f\adj(z)=\ol{f(\ol{z})}$, $\mr{A}\deq A-\brkt{A}$, and
		\beqs
			\brkt{f,g}_{\wh{L}^{2}(\partial\bbD)}\deq\frac{1}{2\pi}\int_{\partial\bbD}(f(z)-f(0))\ol{(g(z)-g(0))}\absv{\dd z}.
		\eeqs
		Furthermore, the mean of $L_{N}(f,A)$ has the following asymptotics for all $\epsilon>0$;
		\beq\label{eq:mean_asymp}
			\expct{L_{N}(f,A)}-N\brkt{A} f(0)=
			\begin{cases}
				O(N^{-1/2+\epsilon}) & \text{if $\chi$ is complex},\\
				\brkt{A}\left(\dfrac{f(1)+f(-1)}{2}-f(0)\right)+O(N^{-1/2+\epsilon}) & \text{if $\chi$ is real}.
			\end{cases}
		\eeq
	\end{thm}
		\begin{rem}
		In particular, for $n=1$ and $A$ and $f$ as above, $L_{N}(f,A)$ has asymptotically the same distribution as the sum of three independent Gaussian random variables $\xi(f,\brkt{A}I),\xi(f,\mr{A}_{\mathrm{d}})$, and $\xi(f,A_{\mathrm{od}})$ where $\mr{A}_{\mathrm{d}}\deq\diag(A_{ii})_{1\leq i\leq N}-\brkt{A}I$ and $A_{\mathrm{od}}\deq A-\diag(A_{ii})_{1\leq i\leq N}$.
	\end{rem}

	\subsection{Gaussian functional representation}\label{sec:repr}
	The goal of this section is to find a representation for $\xi$ as a Gaussian Hilbert space on $H\otimes M_{N}(\C)$, where $H$ is a Hilbert space of test functions and $M_{N}(\C)$ is equipped with the inner product $\brkt{AB\adj}$.
	
	In \cite{Rider-Virag2007IMRN}, it was proved that when $A=I$ the limiting Gaussian process $\xi(f,I)$ can be identified with the \emph{Gaussian free field} conditioned to be harmonic outside the disk. The precise formulation as follows; we consider the Gaussian Hilbert space $\{h(f):f\in H_{0}^{1}(\C)\}$ on the Sobolev space $H_{0}^{1}(\C)$ defined by
	\beqs
		\expct{h(f)}=0,\qquad \expct{h(f)\ol{h(g)}}=\brkt{f,g}_{H_{0}^{1}(\C)},\qquad
		\expct{h(f)h(g)}=\begin{cases}
			0 & \text{if $\chi$ is complex},	\\
			\brkt{f,g\adj}_{H_{0}^{1}(\C)} & \text{if $\chi$ is real},
		\end{cases}
	\eeqs
	where
	\beqs
		H_{0}^{1}(\C)\deq \ol{C_{0}^{\infty}(\C)}^{\norm{\cdot}_{H_{0}^{1}(\C)}},\qquad \brkt{f,g}_{H_{0}^{1}(\C)}=\brkt{\nabla f,\nabla g}_{L^{2}(\C)}.
	\eeqs
	We decompose the Hilbert space $H_{0}^{1}(\C)$ as the direct sum of three subspaces, $H_{0}^{1}(\bbD)$, $H_{0}^{1}(\ol{\bbD}^{c})$, and their complement $H_{0}^{1}(\partial\bbD^{c})^{\perp}$. From \cite[Theorem 1.1 and Corollary 1.2]{Rider-Virag2007IMRN} (see also Section 2.1 of \cite{Cipolloni-Erdos-Schroder2019arXiv}), we have that
	\beqs
		\expct{\absv{\xi(f,I)}^{2}}=\frac{1}{\pi}\norm{f'}_{L^{2}(\bbD)}=\frac{1}{4\pi}\norm{P_{0}f}_{H_{0}^{1}(\C)}^{2},
	\eeqs
	where $P_{0}$ denotes the orthogonal projection onto $H_{0}^{1}(\bbD)\oplus H_{0}^{1}(\partial \bbD^{c})^{\perp}$, the subspace consisting of functions harmonic outside the closed unit disk.
	Therefore $\xi(\cdot,I)$ can be identified with $(4\pi)^{-1/2}P_{0}h$ where $P_{0}h(f)\deq h(P_{0}f)$.
	
	For the functional part in $\xi(f,A)$, we consider the Gaussian field $u$ on the Hardy space $\caH^{2}$ defined by
	\beq\label{eq:GAFunc}
	u(f)=\sum_{k=1}^{\infty}\xi_{k}\wh{f}(k),\qquad f\in \caH^{2},
	\eeq
	 where $\{\xi_{k}:k\in \N\}$ is a collection of i.i.d. standard Gaussian variables in $\C$ if $\chi$ is complex-valued and in $\R$ if $\chi$ is real-valued, and
	\beqs
		\caH^{2}\deq \{f\in L^{2}(\partial\bbD):\wh{f}(k)=0,\,\, \forall k<0\},\qquad \wh{f}(k)\deq \frac{1}{2\pi}\int_{\absv{z}=1}f(z)z^{-k}\absv{\dd z},\quad k\in\Z
	\eeqs
	is considered as a Hilbert subspace of $L^{2}(\partial\bbD)$. For a fixed $A\in M_{N}(\C)$ with $\brkt{A}=0$, we have
	\beqs
		\expct{\absv{\xi(f,A)}^{2}}=\expct{\absv{u(f)}^{2}}\brkt{AA\adj},\qquad \expct{\xi(f,A)^{2}}=\expct{u(f)^{2}}\brkt{AA^{\intercal}}.
	\eeqs
	Note that $\expct{u(f)^{2}}=0$ when $\chi$ is complex.
	
	Now we turn to finding an expression that accounts for the matrix part of $\xi$. Firstly for the tracial part, we consider $P_{0}h\otimes I$ as a random functional on $H_{0}^{1}(\C)\otimes M_{N}(\C)$ given by
	\beq\label{eq:nonf_field}
		(P_{0}h\otimes I)(f\otimes A)\deq h(P_{0}f)\brkt{A}.
	\eeq
	It immediately follows that the covariances of \eqref{eq:nonf_field} for different $(f,A)$'s match the first term of \eqref{eq:cov}. Second, for the functional part, we consider $N^{2}$ i.i.d. copies $\{u_{ij}:i,j\in\bbrktt{N}\}$ of $u$ and define the functional 
	\beqs
		U(f,A)\deq \Tr A\caU(f)\adj, \qquad \caU(f)\in M_{N}(\C),\quad \caU(f)_{ij}\deq \frac{1}{\sqrt{N}}u_{ij}(f).
	\eeqs
	In other words, $\caU(f)$ is an $(N\times N)$ random matrix whose entries are i.i.d. with the same law as $N^{-1/2}u(f)$ for each $f$. Then we have
	\beqs
		\expct{U(f,A)\ol{U(g,B)}}=\brkt{f,g}_{\wh{L}^{2}(\partial\bbD)}\brkt{AB\adj},\quad \expct{U(f,A)U(g,B)}=\begin{cases}
			0 & \text{if $\chi$ is complex},\\
			\brkt{f,g\adj}_{\wh{L}^{2}(\partial\bbD)}\brkt{AB^{\intercal}} & \text{if $\chi$ is real}.
		\end{cases}
	\eeqs
	Therefore, as in \eqref{eq:nonf_field}, we find that
	\beq\label{eq:func_field}
		P_{1}U(f\otimes A)\deq U(P_{1}(f\otimes A))
	\eeq
	has the same covariance as the functional part of \eqref{eq:cov}, where $P_{1}$ denotes the orthogonal projection onto $\caH^{2}\otimes\{A\in M_{N}(\C):\brkt{A}=0\}$. Combining \eqref{eq:nonf_field} and \eqref{eq:func_field}, we can express $\xi$ as the sum of two Gaussian functionals on $(H_{0}^{1}(\C)\cap\caH^{2})\otimes M_{N}(\C)$;
	\beqs
		\xi=\frac{1}{\sqrt{4\pi}}(P_{0}h\otimes I)+P_{1}U.
	\eeqs

\section{Outline of the proof}\label{sec:outline}
	In the rest of the paper, we often omit $N$ to write, for example, $L\equiv L_{N}$, but every quantity should be considered $N$-dependent unless otherwise specified. Also we will focus on the case when $\chi$ is complex, and present how to modify the proof for the real case in Section \ref{sec:real}.

	One of the major obstacles in studying non-Hermitian matrices is that they have complex spectra so that the resolvent no longer has a regularizing effect. A widely used technique to circumvent this problem is Hermitization, whose precise definition is as follows.
	\begin{defn}\label{defn:W}
		For $z\in\C$, we define $W_{z}\in M_{2N}(\C)$ to be the Hermitization of $X-z$, that is,
		\beqs
		\qquad W_{z}\deq\begin{pmatrix}0 & X-z \\ (X-z)\adj & 0\end{pmatrix},
		\eeqs
		and we abbreviate $W\equiv W_{0}$. We further denote $G_{z}(w)\deq(W_{z}-w)^{-1}$ for $w\in\C_{+}$ to be the resolvent of $W_{z}$.
	\end{defn}
	It follows from the Schur complement formula that the resolvent $G_{z}(w)$ can be written in the following block form;
	\begin{align}\label{eq:G_block}
		G_{z}(w)
		&=\begin{pmatrix} 
			w\left((X-z)(X-z)\adj-w^{2}\right)^{-1} & \left((X-z)(X-z)\adj-w^{2}\right)^{-1}(X-z) \\
			(X-z)\adj\left((X-z)(X-z)\adj-w^{2}\right)^{-1} & w\left((X-z)\adj(X-z)-w^{2}\right)^{-1}
		\end{pmatrix}.
	\end{align}
	An immediate consequence of \eqref{eq:G_block} is that $(G_{z}(\ii\eta)^{[12]})^{\adj}=G_{z}(\ii\eta)^{[21]}$.
	
	Note that $W_{z}$ is a Hermitian random matrix with independent entries for which we have more options to approach, and its spectral properties have been studied extensively. In particular, our proof relies on the local law for $W_{z}$ established in \cite{Alt-Erdos-Kruger-Nemish2019}. We remark that the local law for a slightly different Hermitization was proved earlier in \cite{Bourgade-Yau-Yin2014}. In order to rigorously state the local law, we define the deterministic approximation to $G_{z}(\ii\eta)$. For each $\eta>0$, we define $M_{z}(\ii\eta)\in M_{2N}(\C)$ to be the unique solution $M$ of 
	\beq\label{eq:Dyson_M}
	-M\caS[M]=\ii\eta M+M(zI^{\{12\}}+\ol{z}I^{\{21\}})+I,\qquad \caS[P]\deq\begin{pmatrix} \brkt{P^{[22]}} & 0 \\ 0 & \brkt{P^{[11]}}\end{pmatrix},\quad \forall P\in M_{2N}(\C),
	\eeq
	with positive definite imaginary part $\im M=\frac{1}{2\ii}(M-M\adj)$. The equation \eqref{eq:Dyson_M} is an example of \emph{matrix Dyson equation}, so that the existence and uniqueness of its solution follow from \cite{Helton-Rashidi-Speicher2007}. It easy to check from \eqref{eq:Dyson_M} that $M_{z}(\ii\eta)$ is the block constant matrix
	\beqs
	M_{z}(\ii\eta)=\begin{pmatrix}
		m_{z}(\ii\eta) & -zu_{z}(\ii\eta)	\\
		-\ol{z}u_{z}(\ii\eta) & m_{z}(\ii\eta)
	\end{pmatrix},
	\eeqs
	where $m_{z}(\ii\eta)$ and $u_{z}(\ii\eta)$ satisfy
	\beqs
	-\frac{1}{m_{z}(\ii\eta)}=\ii\eta+m_{z}(\ii\eta)-\frac{\absv{z}^{2}}{\ii\eta+m_{z}(\ii\eta)},\quad \im m_{z}(\ii\eta)>0,\qquad 
	u_{z}(\ii\eta)=\frac{m_{z}(\ii\eta)}{\ii\eta+m_{z}(\ii\eta)}.
	\eeqs
	Finally, we recall from \cite[Lemma 3.3]{Alt-Erdos-Kruger2021} that $\norm{M_{z}(\ii\eta)}$ is uniformly bounded in $z$ and $\eta$, that is,
	\beq\label{eq:norm_M}
		\norm{M_{z}(\ii\eta)}+\absv{m_{z}(\ii\eta)}+\absv{u_{z}(\ii\eta)}\lesssim 1.
	\eeq
	
	With these notations, the local law for $W_{z}$ is stated as follows. For simplicity, we here present it only in the relevant regime, that is, outside the spectrum:
	\begin{lem}\cite[Lemma B.1]{Alt-Erdos-Kruger-Nemish2019}\label{lem:ll}
		Let $\delta>0$ be fixed and $A$ be a deterministic $(N\times N)$ matrix with $\norm{A}\leq 1$. Then the following hold uniformly over $\absv{z}^{2}\geq1+\delta$ and $\eta\in(0,1]$;
		\begin{align*}
			\max_{i,j\in\bbrktt{2N}}\absv{(G_{z}(\ii\eta)-M_{z}(\ii\eta))_{ij}}&\prec \frac{1}{\sqrt{N}},\\
			\max_{k,l\in\bbrktt{2}}\Absv{\brkt{A(G_{z}(\ii\eta)-M_{z}(\ii\eta))^{[kl]}}}&\prec \frac{1}{N}.
		\end{align*}
	\end{lem}
	\begin{rem}
		In principle, the  \emph{canonical}  definition of the operator $\caS$
		 in the MDE, given by $P\mapsto \expct{WPW}$,  slightly differs for the real and the complex
		  symmetry class. 
		   When $W$ is the Hermitization of a complex i.i.d. matrix, then  $\expct{WPW}$
		   is given by \eqref{eq:Dyson_M}. However, when 
		   $W$ is the Hermitization of a real i.i.d. matrix $X$, we find that the operator $\wt{\caS}[\cdot]\deq\expct{W\cdot W}$ is given by
		\beq\label{eq:tildeS}
		\wt{\caS}[P]=\begin{pmatrix}
			\brkt{P^{[22]}} & N^{-1}{P^{[21]}}^{\intercal} \\
			N^{-1}{P^{[12]}}^{\intercal} & \brkt{P^{{11]}}}
		\end{pmatrix},
		\eeq
		i.e. it has a tiny off-diagonal part. This would lead to the unique solution of the MDE with $\wt{\caS}$ which is not in a block diagonal form and its entries can mildly depend on $N$ unlike $M_{z}$. To avoid such complication we keep the same matrix Dyson equation with $\caS$ defined in \eqref{eq:Dyson_M}, and its block diagonal solution $M_{z}$ for both the real and complex cases. This fact will play an important role in Section \ref{sec:real} when we prove the main result for real $X$. Further remarks on local laws for non-Hermitian matrices with real entries can be found in \cite[Remark 5.4]{Alt-Erdos-Kruger2021}.
	\end{rem}

	Typically one writes the logarithmic potential of the averaged eigenvalue distribution of $X$ in terms of $G_{z}(\ii\eta)$, via Girko's formula in \cite{Girko1984}. Alternatively, being outside of the spectrum, we can take a rather direct approach to extract the eigenvalues of $X$ from $W_{z}$. We will use that whenever $(X-z)(X-z)\adj$ is invertible, it follows that
	\beq\label{eq:eta_limit}
	\lim_{\eta\to0}G_{z}^{[21]}(\ii\eta)=(X-z)\adj\lim_{\eta\to0}\frac{1}{(X-z)(X-z)\adj+\eta^{2}} = \frac{1}{(X-z)}.
	\eeq
	In fact, on the event $\Omega_{\delta}(\kappa)$ we can quantify \eqref{eq:eta_limit} with a concrete rate of convergence in terms of $\eta$. Recall from Lemma \ref{lem:sing} that the smallest singular value of $(X-z)$ is bounded from below by $\kappa$ on $\Omega_{\delta}(\kappa)$, which has very high probability. Since the eigenvalues of $W_{z}$ have the same modulus as the singular values of $(X-z)$, we have
	\beq\label{eq:normG}
		\norm{G_{z}(\ii\eta)}\leq (\norm{((X-z)(X-z)\adj)^{-1}}+\eta^{2})^{-1/2}\leq \kappa^{-1}\qquad \text{on }\,\Omega_{\delta}(\kappa),
	\eeq
	uniformly over $\absv{z}\geq 1+\delta/2$ and $\eta>0$. Using \eqref{eq:G_block}, this further implies
	\beqs
		\Norm{G^{[21]}_{z}(\ii\eta)-(X-z)^{-1}}
		=\Norm{\left(\frac{\eta^{2}}{((X-z)(X-z)\adj+\eta^{2})(X-z)}\right)}\leq \kappa^{-3}\eta^{2}
	\eeqs
	uniformly over $\absv{z}\geq1+\delta/2$, so that on $\Omega_{\delta}(\kappa)$ we have, for any $A\in M_{N}(\C)$ with $\norm{A}\leq 1$, that
	\beq\label{eq:G_approx}
		\absv{\Tr ((X-z)^{-1}A)- \Tr G^{[21]}_{z}(\ii\eta)A}	\leq \kappa^{-3}N\eta^{2}
	\eeq
	
	As a consequence of Lemma \ref{lem:sing} and \eqref{eq:G_approx}, we can prove that the difference between $\ctr{L(f,A)}$ and the contour integral of $\ctr{\Tr G^{[21]}_{z}(\ii\eta)A}$ is negligible when $\eta$ is small. We here emphasize that the latter is not restricted on the event $\Omega_{\delta}(\kappa)$ unlike the former.
	\begin{prop}\label{prop:toG}
		Let $\eta=N^{-2}$ and suppose that the assumptions of Theorem \ref{thm:main} hold true. Then we have
		\beq\label{eq:Hermite}
			(\ctr{L_{N}(f_{i},A_{i})})_{i\in\bbrktt{n}}\overset{m}{=}-\left(\frac{1}{2\pi\ii}\oint_{\gamma}f_{i}(z)\ctr{\Tr G_{z_{i}}(\ii\eta)^{[21]}A_{i}}\dd z\right)_{i\in\bbrktt{n}}	+O_{\rmm}(N^{-3+\epsilon}),
		\eeq
		where $\gamma$ is the positively oriented circle $\{z:\absv{z}=1+\delta/2\}$. 
	\end{prop}
	\begin{proof}
	For simplicity, we prove \eqref{eq:Hermite} when $n=1$ since the proof for multi-dimensional vectors is identical except for a few minor algebraic modification. Also, we fix the parameters $\delta,\kappa>0$ and will not keep track of their effect on convergence rates.
	
	Taking the Cauchy integral of \eqref{eq:G_approx} over the path $\gamma=\{\absv{z}=1+\delta/2\}$ leads to
	\beq\label{eq:mean_int}
		L(f,A)=\lone_{\Omega_{\delta}(\kappa)}\frac{1}{2\pi\ii}\oint_{\gamma}f(z)\Tr((X-z)^{-1}A)\dd z
		=\lone_{\Omega_{\delta}(\kappa)}\frac{1}{2\pi\ii}\oint_{\gamma}f(z)\Tr(G_{z}(\ii\eta)^{[21]}A)\dd z +O(N\eta^{2}),
	\eeq
	so that the same equality holds with both sides centered. We also note from Lemma \ref{lem:ll} and $\norm{M_{z}(\ii\eta)}\lesssim 1$ that $\ctr{\lone_{\Omega_{\delta}(\kappa)}\Tr G_{z}(\ii\eta)^{[21]}A}\prec 1,$ 
	which, together with \eqref{eq:normG}, gives for any $p,\epsilon,D>0$ that
	\beq\label{eq:finite_mom}
		\expct{\absv{\ctr{\lone_{\Omega_{\delta}(\kappa)}\Tr G_{z}(\ii\eta)^{[21]}A}}^{p}}\leq N^{\epsilon}+N^{p-D}.
	\eeq
	By our choice of $\eta=N^{-2}$ we have $N\eta^{2}=N^{-3}$, so that from $\absv{x^{p}-y^{p}}\leq\sum_{k=1}^{p}\binom{p}{k}\absv{x-y}^{k}\absv{x}^{p-k}$ we get
	\begin{multline*}
		\Absv{\Expct{\ctr{L(f,A)}^{p}-\left(\frac{1}{2\pi\ii}\oint_{\gamma}f(z)\ctr{\lone_{\Omega_{\delta}(\kappa)}\Tr G_{z}(\ii\eta)^{[21]}A}\dd z\right)^{p}}} \\
		\lesssim \eta^{2}\int_{\gamma}\expct{\absv{\ctr{\lone_{\Omega_{\delta}(\kappa)}\Tr G_{z}(\ii\eta)^{[21]}A}}^{p}}\absv{\dd z}=O(N^{-3+\epsilon}),
	\end{multline*}
	where we applied H\H{o}lder's inequality to interchange the contour integral with the $L^{p}$ norm. Similar estimate holds for generic $*$-monomials of the form $L^{p}\ol{L}^{q}$. In other words, we have shown that
	\beq\label{eq:L_iint_1}
		\ctr{L(f,A)}\overset{m}{=} -\frac{1}{2\pi\ii}\oint_{\gamma}f(z)\ctr{\lone_{\Omega_{\delta}(\kappa)}\Tr G_{z}(\ii\eta)^{[21]}A}\dd z+ O_{\rmm}(N^{-3+\epsilon})
	\eeq
	for all $\epsilon>0$.
	
	On the other hand, by Lemma \ref{lem:sing} we have for all $p>0$ that
	\beq\label{eq:event_del}
		\expct{\absv{\lone_{\Omega_{\delta}(\kappa)^{c}}\Tr G_{z}(\ii\eta)^{[21]}A}^{p}}=O(N^{p-D}\eta^{-p})=O(N^{3p-D}),
	\eeq
	where we used the trivial bound $\norm{G}\leq \eta^{-1}$. Enlarging $D$ depending on $p$ leads to
	\beq\label{eq:event_del_conc}
		\lone_{\Omega_{\delta}(\kappa)^{c}}\frac{1}{2\pi\ii}\oint_{\gamma}f(z)\Tr (G_{z}(\ii\eta)^{[21]}A)\dd z \overset{m}{=}O_{\rmm}(N^{-D})
	\eeq
	for any fixed $D$. Then \eqref{eq:event_del_conc} together with \eqref{eq:finite_mom} gives, after a binomial expansion, that
	\beq\label{eq:L_iint_2}
		\frac{1}{2\pi\ii}\oint_{\gamma}\ctr{\lone_{\Omega_{\delta}(\kappa)}\Tr G_{z}(\ii\eta)^{[21]}A}\dd z
		\overset{m}{=}\frac{1}{2\pi\ii}\oint_{\gamma}\ctr{\Tr G_{z}(\ii\eta)^{[21]}A}\dd z+O_{\rmm}(N^{-D})
	\eeq
	for all fixed $D>0$. Combining \eqref{eq:L_iint_1} and \eqref{eq:L_iint_2}, we conclude \eqref{eq:Hermite} for any fixed $\epsilon>0$
	\end{proof}
	\begin{rem}
		As easily seen from the proof, our choice of $\eta=N^{-2}$ in Proposition \ref{prop:toG} is purely cosmetic. In fact, the same proof applies whenever $\log \eta/\log N\in (-\infty,-1/2)$, in which case the upper bound in \eqref{eq:Hermite} becomes $O_{\rmm}(N^{1+\epsilon}\eta^{2})$.
	\end{rem}
	In the following sections, we will prove that for $\eta=N^{-2}$, the random vector
	\beqs
		-\left(\frac{1}{2\pi\ii}\oint_{\gamma}f_{i}(z)\Tr(G_{z}(\ii\eta)^{[21]}A_{i})\dd z\right)_{i\in\bbrktt{n}}
	\eeqs
	has asymptotically the same moment as the Gaussian vector $\xi(f_{i},A_{i})$, which would imply Theorem \ref{thm:main} via Proposition \ref{prop:toG}. In this regard, for the rest of the paper we fix $\eta=N^{-2}$ and we suppress the spectral parameter to write $G_{z}\equiv G_{z}(\ii\eta)$, $M_{z}\equiv M_{z}(\ii\eta)$ et cetera. We recall from \cite[Eq. (3.7)]{Cipolloni-Erdos-Schroder2019arXiv} that the deterministic matrix $M_{z}$ has the following asymptotic expansions in $\eta$ for $\absv{z}>1$;
	\beq\label{eq:asymp_M}
		\Norm{M_{z}-\begin{pmatrix}
			\ii\eta(\absv{z}^{2}-1)^{-1} & -\ol{z}^{-1} \\ -{z}^{-1} & \ii\eta(\absv{z}^{2}-1)^{-1} \end{pmatrix}}=O(\eta^{2}).
	\eeq

	For $\eta$ fixed to be $N^{-2}$, we state our main technical result, the central limit theorem for the resolvents $(\Tr G_{z_{i}}^{[21]}A_{i})_{i\in\bbrktt{n}}$ as follows. Its proof for complex and real $\chi$ is postponed to Sections \ref{sec:prf_CLT_G} and $\ref{sec:real}$, respectively.
	\begin{prop}\label{prop:CLT_G}
		Let $\epsilon,\delta>0$ be fixed and take a finite collection of deterministic matrices $A_{1},\cdots,A_{n}\in M_{N}(\C)$ with $\norm{A_{i}}\leq 1$ and $z_{1},\cdots,z_{n}\in\bbD_{1+\delta/2}^{c}$. Then, for $\eta=N^{-2}$, we have
		\begin{align}\label{eq:CLT_G}
			(\ctr{\Tr G_{z_{i}}^{[21]}A_{i}})_{i\in\bbrktt{n}}
			&\overset{m}{=} (\zeta(z_{i},A_{i}))_{i\in\bbrktt{n}}+O_{\rmm}(N^{-1/2+\epsilon}),\\
			\label{eq:mean_asymp_G}
			\expct{\Tr G_{z_{i}}^{[21]}A_{i}}+N\brkt{A_{i}}\frac{1}{z_{i}}&=
			\begin{cases}
				O(N^{-1/2+\epsilon}) & \text{if $\chi$ is complex},\\
				-\dfrac{\brkt{A_{i}}}{z_{i}(z_{i}^{2}-1)}+O(N^{-1/2+\epsilon}) & \text{if $\chi$ is real},
			\end{cases}	
		\end{align}
		where $\{\zeta(z,A):\absv{z}>1,A\in M_{N}(\C)\}$ is a centered Gaussian process with covariances
		\beqs
		\begin{split}
			\expct{\zeta(z,A)\ol{\zeta(w,B)}}&=\brkt{A}\ol{\brkt{B}}\frac{1}{(1-z\ol{w})^{2}}+\brkt{\mr{A}\mr{B}\adj}\frac{1}{z\ol{w}(z\ol{w}-1)},\\ \expct{\zeta(z,A)\zeta(w,B)}&=\begin{cases}
				0 & \text{if $\chi$ is complex,}\\
				\brkt{A}\brkt{B}\dfrac{1}{(1-zw)^{2}}+\brkt{\mr{A}\mr{B}^{\intercal}}\dfrac{1}{zw(zw-1)} & \text{if $\chi$ is real}.
			\end{cases}
		\end{split}
		\eeqs
		Furthermore, the convergences in \eqref{eq:CLT_G} and \eqref{eq:mean_asymp_G} are uniform over $z_{1},\cdots,z_{n}\in\bbD_{1+\delta/2}^{c}$.
	\end{prop}

	Now that we have collected all the necessary ingredients, we complete the proof of Theorem \ref{thm:main}.	
	\begin{proof}[Proof of Theorem \ref{thm:main}]
		First of all, by \eqref{eq:Hermite} we have
		\beqs
			(\ctr{L(f_{i},A_{i})})_{1\leq i\leq n}\overset{m}{=}\left(\frac{1}{2\pi \ii}\oint_{\gamma}f_{i}(z)\ctr{\Tr G_{z}^{[21]}A_{i}}\dd z\right)_{1\leq i\leq n}+O_{\rmm}(N^{-3+\epsilon}).
		\eeqs
		Then, for a fixed $*$-monomial $\prod_{i=1}^{k}x_{j_{i}}^{m_{i}}$ with a multiset $\{j_{1},\cdots,j_{k}\}$ of indices in $\bbrktt{1,n}$ and $m_{i}\in\{1,*\}$, we have
		\beq\label{eq:thm_prf}
			\oint_{\gamma}\cdots\oint_{\gamma}\Expct{\prod_{i=1}^{k}(f_{j_{i}}(z_{{i}})\ctr{\Tr G_{z_{{i}}}^{[21]}A_{j_{i}}})^{m_{i}}\dd (z_{{i}}^{m_{i}})}	
			=\oint_{\gamma}\cdots\oint_{\gamma}\Expct{\prod_{i=1}^{k}(f_{j_{i}}(z_{{i}})\zeta(z_{{i}},A_{j_{i}}))^{m_{i}}\dd (z_{{i}}^{m_{i}})}+O(N^{-3+\epsilon}),
		\eeq
		where we used that the estimate in \eqref{eq:CLT_G} is uniform over $z_{i}$'s in $\{z:\absv{z}\geq 1+\delta/2\}$. To calculate the right-hand side of \eqref{eq:thm_prf}, we use Wick's theorem and the fact that (see Eq. (1.7) in \cite{Rider-Silverstein2006} for a proof of the last identity)
		\beqs
		\begin{split}
			-\frac{1}{4\pi^{2}}\oint_{\gamma}\oint_{\gamma}\frac{f(z)\ol{g(w)}}{z\ol{w}(z\ol{w}-1)}\dd\ol{w}\dd z
			&=\frac{1}{2\pi}\int_{\partial\bbD}f(z)\ol{g(z)}\dd z-f(0)\ol{g(0)}=\brkt{f,g}_{\wh{L}^{2}(\partial \bbD)},
			\\
			-\frac{1}{4\pi^{2}}\oint_{\gamma}\oint_{\gamma}\frac{f(z)g(w)}{zw(zw-1)}\dd w\dd z
			&=\brkt{f,\ol{g}}_{\wh{L}^{2}(\partial\bbD)},
		\end{split}
		\eeqs
		Thus the covariance of contour integrals of $\zeta(\cdot,A)$ coincide with that of $\xi(f,A)$, hence combining with \eqref{eq:thm_prf} leads to
		\beqs
			(2\pi\ii)^{-k}\oint_{\gamma}\cdots\oint_{\gamma}\Expct{\prod_{i=1}^{k}(f(z_{i})\ctr{\Tr G_{z_{j_{i}}}^{[21]}A_{j_{i}}})^{m_{i}}\dd (z^{m_{i}})}
			=\Expct{\prod_{i=1}^{k}\xi(f_{j_{i}},A_{j_{i}})^{m_{i}}}+O(N^{-1/2+\epsilon}),
		\eeqs
		which implies \eqref{eq:main}. 
		
		In order to prove \eqref{eq:mean_asymp}, we use \eqref{eq:mean_int} and \eqref{eq:event_del} with $p=1$ to get
		\beqs
			\expct{L(f,A)}=\frac{1}{2\pi\ii}\oint_{\gamma}f(z)\expct{\Tr G_{z}^{[21]}A}\dd z +O(N^{-3}).
		\eeqs
		Therefore \eqref{eq:mean_asymp_G} implies \eqref{eq:mean_asymp} as desired. This concludes the proof of Theorem \ref{thm:main}.
	\end{proof}
	
	We conclude this section with remarks on Proposition \ref{prop:CLT_G}.
	\begin{rem}
		As in Proposition \ref{prop:toG}, the choice of $\eta=N^{-2}$ does not play a crucial role in the sense that the same proof applies to smaller $\eta$ as long as $\log\eta/\log N$ remains bounded; for example, we may take $\eta=N^{-100}$. However, taking smaller $\eta$ does not improve the rate of convergence in \eqref{eq:CLT_G} or \eqref{eq:mean_asymp_G}, in contrast to Proposition \ref{prop:toG}.
	\end{rem}

\section{Proof of Proposition \ref{prop:CLT_G}}\label{sec:prf_CLT_G}
	
	In this section, we prove that the joint moments of $\ctr{\Tr G_{z}^{[21]}A}$ converge to those of corresponding Gaussian variables, and the asymptotics \eqref{eq:mean_asymp_G} of the mean will be proved in Section \ref{sec:mean_prf}. 
	We first prove Proposition \ref{prop:CLT_G} when $\chi$ is complex, and the proof for real-valued $\chi$ will be given in Section \ref{sec:real}.
	
	For any $p,q\in\N$, we calculate the moment
	\beq\label{eq:j_mmnt_goal}
		\bbE\bigg[\prod_{i\in\bbrktt{p}}\caX_{i}\prod_{j\in\bbrktt{q}}\caY_{j}\bigg],
	\eeq
	where we defined
	\begin{align}\label{eq:defn_XY}
		&\caX_{i}=\ctr{\Tr G_{z_{i}}^{[21]}A_{i}},\quad i\in\bbrktt{p}, &
		&\caY_{j}=\ctr{\Tr (G_{w_{j}}^{[21]})\adj B_{j}\adj}=\ctr{\Tr G_{w_{j}}^{[12]}B_{j}\adj},\quad j\in\bbrktt{q},
	\end{align}
	and abbreviated $G_{z}\equiv G_{z}(\ii\eta)$. We further assume that there exist sets of indices $S\subset\bbrktt{p}$ and $T\subset\bbrktt{q}$ such that 
	\beqs
		\Tr A_{i}=0=\Tr B_{j}, \quad i\in S,\,j\in T,\qquad 
		A_{i}=I=B_{j},\quad i\in S^{c},\, j\in T^{c},
	\eeqs
	where $S^{c}=\bbrktt{p}\setminus S$ and $T^{c}=\bbrktt{q}\setminus T$.
	We remark that calculating all joint moments as in \eqref{eq:j_mmnt_goal} is equivalent to that of Proposition \ref{prop:CLT_G}, since for any $(N\times N)$ matrix $A$ we may decompose $\Tr G_{z}^{[21]}A$ as
	\beqs
		\Tr G^{[21]}A= \Tr G^{[21]}\mr{A}+\brkt{A}\Tr G^{[21]}
	\eeqs
	so that the first and second terms correspond respectively to the case $i\in S$ and $i\in S^{c}$. We further remark that it is necessary to separate the complex conjugates $\caY$ from $\caX$ in contrast to \cite{Cipolloni-Erdos-Schroder2019arXiv}, for $\ol{\Tr G^{[21]}A}$ no longer has the same form as $\Tr (G\adj)^{[21]} A$ since $(G^{[21]})\adj=(G\adj)^{[12]}=G^{[12]}$. To simplify the presentation, we abbreviate
	\begin{align*}
		&\caX\deq \prod_{i\in\bbrktt{p}}\caX_{i},&
		&\caX^{(i_{1},\cdots,i_{n})}\deq \prod_{\substack{i\in \bbrktt{p},\\i\neq i_{1},\cdots,i_{n}}}\caX_{i},
	\end{align*}
	and define $\caY$ similarly. With these notations, we have the following asymptotic Wick formula;
	\begin{prop}\label{prop:Wick}
		Let $\delta,\epsilon>0$ be fixed and suppose that $z_{i},w_{j}\in \bbD_{1+\delta}^{c}$ for all $i,j$. Then we have
		\beq\label{eq:Wick_result}
			\expct{\caX\caY}=\sum_{\substack{P\in\mathrm{Pair}(S,T) \\ Q\in\mathrm{Pair}(S^{c},T^{c})}}\prod_{(i,j)\in P}V^{\circ}(z_{i},w_{j})\brkt{A_{i}B_{j}\adj}\prod_{(i,j)\in Q}V(z_{i},w_{j})+O(N^{-1/2+\epsilon}),
		\eeq
		where $\mathrm{Pair}(I,J)$ for $I\subset\bbrktt{p}$ and $J\subset\bbrktt{q}$ denotes the set of perfect matchings from $I$ to $J$ and 
		\beq\label{eq:defn_V}
			V^{\circ}(z,w)=\frac{1}{z\ol{w}(z\ol{w}-1)},\qquad V(z,w)=\frac{1}{(1-z\ol{w})^{2}}.
		\eeq
		In particular, $\expct{\caX\caY}$ is nonzero only if $\absv{S}=\absv{T}$ and $\absv{S^{c}}=\absv{T^{c}}$, hence $p=q$.
	\end{prop}
	
	Before proceeding to the proof of Proposition \ref{prop:Wick}, we remark that the averaged local law gives a priori bounds on $\Tr G_{z}^{[kl]}A$ for $k,l\in\{1,2\}$. To see this, note that Lemma \ref{lem:ll} and \eqref{eq:asymp_M} imply
	\beqs
		\absv{\Tr G_{z}^{[kl]}A+\brkt{A}\Tr M_{z}^{[kl]}}=\absv{\Tr (G_{z}-M_{z})A^{\{lk\}}}\prec 1,
	\eeqs
	and by triangle inequality this leads to $\absv{\ctr{\Tr G_{z}^{[kl]}A}}\prec 1$. In particular, we have
	\beq\label{eq:prior}
		\absv{\Tr G^{[kl]}A}\prec 1 \qquad	\text{if }\brkt{A}=0\text{ or }k=l,
	\eeq
	where we used \eqref{eq:asymp_M} and $\eta=N^{-2}$ for the case $k=l$. The same argument applies to $\caX_{i}$ and $\caY_{j}$, so that 
	\beq\label{eq:prior_XY}
		\caX_{i}=\ctr{\Tr G_{z_{i}}^{[21]}A_{i}}\prec 1, \qquad \caY_{j}=\ol{\ctr{\Tr G_{w_{j}}^{[21]}B_{j}}}\prec 1.
	\eeq
	We will often use \eqref{eq:prior_XY} to simply replace factors of $\caX_{i}$ and $\caY_{j}$ by $N^{\epsilon}$ in various estimates.
	
	\begin{proof}[Proof of Proposition \ref{prop:Wick}]
		We aim to derive asymptotic Wick formulas for the joint moment in \eqref{eq:Wick_result} with respect to the first index $1\in\bbrktt{p}$. To be precise, we prove that
		\begin{align}
			\expct{\caX\caY}&=\sum_{j\in T}\frac{\brkt{A_{1}B_{j}\adj}}{z_{1}\ol{w}_{j}(z_{1}\ol{w}_{j}-1)}\expct{\caX^{(1)}\caY^{(j)}}+O(N^{-1/2+\epsilon})&	&\text{if }1\in S, \label{eq:Wick_func}\\
			\expct{\caX\caY}&=\sum_{j\in T^{c}}\frac{1}{(z_{1}\ol{w}_{j}-1)^{2}}\expct{\caX^{(1)}\caY^{(j)}}+O(N^{-1/2+\epsilon})&	&\text{if }1\in S^{c}	\label{eq:Wick_nonf}
		\end{align}
	hold uniformly over $z_{i},w_{j}\in\bbD_{1+\delta}^{c}$.
		
	Using the resolvent identity $(W_{z}-\ii\eta)G=I$, we write 
	\beqs
		\ii\eta\Tr G_{z_{1}}^{[11]}A_{1}=\Tr XG_{z_{1}}^{[21]} A-z\Tr G_{z_{1}}^{[21]}A_{1}-\Tr A_{1}.
	\eeqs
	Since $\Tr A_{1}$ is deterministic, the identity $\expct{X\ctr{Y}}=\expct{\ctr{X}Y}$ gives
	\beq\label{eq:j_mmnt_1}
	\begin{split}
		\expct{\caX \caY}
		=\expct{(\Tr G_{z_{1}}^{[21]}A_{1})\ctr{\caX^{(1)}\caY}}
		&=\frac{\ol{z}_{1}}{\absv{z_{1}}^{2}}\expct{\Tr \left((XG_{z_{1}}^{[21]}-\ii\eta G_{z_{1}}^{[11]})A_{1}\right)\ctr{\caX^{(1)}\caY}}\\
		&=\frac{1}{z_{1}}\expct{\Tr XG_{z_{1}}^{[21]}\ctr{\caX^{(1)}\caY}}+O(N^{\epsilon}\eta),	
	\end{split}
	\eeq
	where we used \eqref{eq:prior} to the second term.
	
	Now we calculate the first term on the right-hand side of \eqref{eq:j_mmnt_1} using cumulant expansions (see for example \cite{He-Knowles2017});
	\beq\label{eq:Stein_2}
	\begin{split}
		&\expct{\Tr({XG_{z_{1}}^{[21]} A_{1}})\ctr{\caX^{(1)}\caY}}	=\sum_{a,b}\expct{X_{ab}\Tr\Delta_{ab}G_{z_{1}}^{[21]}A_{1}\ctr{\caX^{(1)}\caY}}
		=\frac{1}{N}\sum_{a,b}\Expct{\Tr \Delta_{ab}\partial_{ab}^{(0,1)}[G_{z_{1}}^{[21]}]A_{1}\ctr{\caX^{(1)}\caY}}\\
		+&\frac{1}{N}\sum_{i\in\bbrktt{2,p}}\sum_{a,b}\Expct{\partial_{ab}^{(0,1)}[\caX_{i}]\Tr\Delta_{ab}G_{z_{1}}^{[21]}A_{1}\caX^{(1,i)}\caY}
		+\frac{1}{N}\sum_{j\in\bbrktt{q}}\sum_{a,b}\Expct{\partial_{ab}^{(0,1)}[\caY_{j}]\Tr\Delta_{ab}G_{z_{1}}^{[21]}A_{1}\caX^{(1)}\caY^{(j)}}	\\
		+&\sum_{\substack{k,l\in\N \\ k+l\geq 2}}\frac{\kappa(k+1,l)}{k!l!}\sum_{a,b}\Expct{\partial^{(k,l)}_{ab}\left(\Tr\Delta_{ab}G_{z_{1}}^{[21]}A_{1}\ctr{\caX^{(1)}\caY}\right)},
	\end{split}
	\eeq
	where we defined $\Delta_{ab}\in M_{N}(\C)$ with $(\Delta_{ab})_{ij}=\delta_{ai}\delta_{bj}$, the cumulants
	\beqs
	\kappa(k,l)\deq (-\ii)^{k+l}\left[\frac{\partial^{k+l}}{\partial s^{k}\partial t^{l}}\log\expct{\e{\ii sX_{11}+\ii t\ol{X}_{11}}}\right]_{s,t=0},	\AND \partial^{(k,l)}_{ab}\deq\frac{\partial^{k+l}}{\partial X_{ab}^{k}\partial\ol{X}_{ab}^{l}}.
	\eeqs
	Note that the term in \eqref{eq:Stein_2} corresponding to $\partial_{ba}^{(1,0)}$ vanishes since $\kappa(2,0)=\expct{X_{11}^{2}}=0$. Applying $\partial_{ab}^{(0,1)}[G^{[kl]}_{z_{1}}]=-G_{z_{1}}^{[k2]}\Delta_{ba}G_{z_{1}}^{[1l]}$ to the first term of \eqref{eq:Stein_2}, we find that
	\beqs
		\frac{1}{N}\sum_{a,b}\Tr \Delta_{ab}\partial_{ab}^{(0,1)}[G_{z_{1}}^{[21]}]A_{1}=-\frac{1}{N}\sum_{a,b}\Tr \Delta_{ab}G_{z_{1}}^{[22]}\Delta_{ba}G_{z_{1}}^{[11]}A_{1}=-\brkt{G_{z_{1}}^{[22]}}\Tr G_{z_{1}}^{[11]}A_{1}.
	\eeqs
	Thus the contribution of the first term on the right-hand side of \eqref{eq:Stein_2} reads
	\beq\label{eq:Stein_1_1}
	\begin{split}
		&\absv{\expct{\brkt{G_{z_{1}}^{[22]}}\Tr G_{z_{1}}^{[11]}A_{1}\ctr{\caX^{(1)}\caY}}}
		=\frac{1}{N}\absv{\expct{(\Tr G_{z_{1}}^{[22]})(\Tr G_{z_{1}}^{[11]}A_{1})\ctr{\caX^{(1)}\caY}}}
		\leq N^{-1+\epsilon},
	\end{split}
	\eeq
	where we applied \eqref{eq:prior} to the first two traces and \eqref{eq:prior_XY} to get $\caX,\caY\prec 1$.
	
	The same calculations for the second and third terms of \eqref{eq:Stein_2} give
	\beq\label{eq:Stein_1_2}
	\begin{split}
		&\frac{1}{N}\sum_{a,b}\partial_{ab}^{(0,1)}[\caX_{i}]\Tr\Delta_{ab}G_{z_{1}}^{[21]}A_{1}
		=-\brkt{G_{z_{i}}^{[11]}A_{i}G_{z_{i}}^{[22]}G_{z_{1}}^{[21]}A_{1}}=:-\frX_{i},\\
		&\frac{1}{N}\sum_{a,b}\partial_{ab}^{(0,1)}[\caY_{j}]\Tr\Delta_{ab}G_{z_{1}}^{[21]}A_{1}
		=-\brkt{G_{w_{j}}^{[12]}B_{j}\adj G_{w_{j}}^{[12]}G_{z_{1}}^{[21]}A_{1}}=:-\frY_{j}.
	\end{split}
	\eeq
	Substituting \eqref{eq:Stein_1_1} and \eqref{eq:Stein_1_2} into \eqref{eq:Stein_2} and then plugging the result back to \eqref{eq:j_mmnt_1}, we conclude
	\beq\label{eq:Wick_start}
	\begin{split}
		z_{1}\expct{\caX\caY}	=&-\sum_{i\in\bbrktt{2,p}}\expct{\frX_{i}\caX^{(1,i)}\caY}-\sum_{j\in\bbrktt{q}}\expct{\frY_{j}\caX^{(1)}\caY^{(j)}}	\\
		&+\sum_{\substack{k,l\in\Z_{+}\\ k+l\geq 2}}\frac{\kappa(k+1,l)}{k!l!}\sum_{a,b}\Expct{\partial^{(k,l)}_{ab}\left(\Tr\Delta_{ab}G_{z_{1}}^{[21]}A_{1}\ctr{\caX^{(1)}\caY}\right)}+O(N^{-1+\epsilon}).
	\end{split}
	\eeq
	
	We estimate each term of \eqref{eq:Wick_start} using the following proposition to derive \eqref{eq:Wick_func} and \eqref{eq:Wick_nonf}. Its proof is postponed to Section \ref{sec:prf_tech}.
	\begin{prop}\label{prop:tech}
		Under the assumptions of Proposition \ref{prop:Wick}, the following estimates hold uniformly over $z_{i},w_{j}\in\bbD_{1+\delta}^{c}$ and for any fixed $k,l\in\N$ with $k+l\geq 2$;
	\begin{align}
		\expct{\absv{\frX_{i}}^{2}}=O(N^{-2+\epsilon}), \qquad
		\Expct{\Absv{\frY_{j}+\frac{\brkt{A_{1}B_{j}\adj}}{\ol{w}_{j}(z_{1}\ol{w}_{j}-1)}+\frac{z_{1}\brkt{A_{1}}\brkt{B_{j}\adj}}{(z_{1}\ol{w}_{j}-1)^{2}}}^{2}}&=O(N^{-2+\epsilon}),	\label{eq:comput1}\\
		\frac{\kappa(k+1,l)}{k!l!}\sum_{a,b}\partial^{(k,l)}_{ab}\left[\Tr\Delta_{ab}G_{z_{1}}^{[21]}A_{1}\ctr{\caX^{(1)}\caY}\right]\prec N^{-1/2-(k+l-4)_{+}/2}.&	\label{eq:comput2}
	\end{align} 
	\end{prop}

	Armed with Proposition \ref{prop:tech}, we proceed with the proof of Proposition \ref{prop:Wick}. We use Cauchy-Schwarz inequality to estimate
	\beqs
	\absv{\expct{\frX_{i}\caX^{(1,i)}\caY}}\leq \expct{\absv{\frX_{i}}^{2}}^{1/2}\expct{\absv{\caX^{(1,i)}\caY}^{2}}^{1/2}\leq N^{-1+\epsilon}
	\eeqs
	where we used \eqref{eq:prior_XY} and \eqref{eq:comput1} in the second inequality, and we can estimate the subleading term of $\frY_{j}$ similarly. Combining \eqref{eq:Wick_start}, \eqref{eq:comput1}, and \eqref{eq:comput2} proves \eqref{eq:Wick_func} and \eqref{eq:Wick_nonf}.
	
	Since the specific choice of index $1$ in \eqref{eq:Wick_func} and \eqref{eq:Wick_nonf} played no role in the argument, we can apply \eqref{eq:Wick_func} and \eqref{eq:Wick_nonf} iteratively for any index, yielding \eqref{eq:Wick_result}. This completes the proof of Proposition \ref{prop:Wick} for complex-valued $\chi$, modulo Proposition \ref{prop:tech}
	\end{proof}

	\section{Proof of Proposition \ref{prop:tech} and \eqref{eq:mean_asymp_G}}\label{sec:prf_tech}
	The quantities $\frX_{i}$ and $\frY_{j}$ are traces of products of three resolvent, whose variants were studied in \cite{Cipolloni-Erdos-Schroder2019arXiv}. We will shortly see that similar calculations apply to $\frX_{i}$ and $\frY_{j}$, and in fact the local laws for product of two resolvent, Theorem 5.2 therein, can be used verbatim. To state the local laws for products, we introduce the relevant notations. For given $z,w\in\C$, we define the operator $\wh{\caB}(z,w)$ as 
	\beqs
		\wh{\caB}(z,w)\deq 1-M_{z}(\ii\eta)\caS[\cdot]M_{w}(\ii\eta),
	\eeqs
	where $\caS$ was defined in \eqref{eq:Dyson_M}. For a given deterministic $(2N\times 2N)$ matrix $B$, we further define
	\beqs
		M_{B}(z,w)=\wh{\caB}(z,w)^{-1}[M_{z}(\ii\eta)BM_{w}(\ii\eta)].
	\eeqs
	Recall from \cite[Eq. (6.4)]{Cipolloni-Erdos-Schroder2019arXiv} that the operator $\wh{\caB}(z,w)$ has four eigenvalues, $1,1$, $\wh{\beta}$, and $\wh{\beta}_{*}$, given by
	\beq\label{eq:def_beta}
		\wh{\beta},\wh{\beta}_{*}\deq 1-u_{z}u_{w}\re (z\ol{w})\pm\sqrt{m_{z}^{2}m_{w}^{2}-u_{z}^{2}u_{w}^{2}(\im (z\ol{w}))^{2}}.
	\eeq
	Applying the asymptotics in \eqref{eq:asymp_M}, for $\absv{z}^{2},\absv{w}^{2}>1+\delta$ we find that
	\beq\label{eq:beta_lb}
		\wh{\beta}\wh{\beta}_{*}		=\absv{u_{z}u_{w}z\ol{w}-1}^{2}-m_{z}^{2}m_{w}^{2}\geq\absv{(1+\delta)^{-1}-1}+O(\eta^{2})\gtrsim\delta.
	\eeq
	Since $\absv{\beta},\absv{\beta_{*}}\lesssim1$ the inverse of $\wh{\caB}$ has $O(1)$ operator norm, in contrast to Lemma 6.1 of \cite{Cipolloni-Erdos-Schroder2019arXiv} where the parameters $z,w$ were allowed inside the unit disk. In particular, this allows us to strengthen the bound of \cite[Theorem 5.2]{Cipolloni-Erdos-Schroder2019arXiv} in the following form. 
	\begin{lem}\label{lem:2Gll}
		Let $\delta>0$, $P,Q\in M_{2N}(\C)$ and $\bsx,\bsy\in \C^{2N}$ be deterministic with $\norm{P},\norm{Q},\norm{\bsx},\norm{\bsy}\leq 1$. Then the following hold uniformly over $z,w\in \{z\in\C:\absv{z}^{2}>1+\delta\}$;
		\beqs
		\begin{aligned}
			\absv{\brkt{\bsx,(G_{z}P G_{w}-M_{P}(z,w))\bsy}}&\prec \frac{1}{\sqrt{N}}, \\
			\absv{\brkt{Q(G_{z}PG_{w}-M_{P}(z,w))}}&\prec \frac{1}{N},
		\end{aligned}
		\eeqs
		where we abbreviated $G_{z}\equiv G_{z}(\ii\eta)$ and $G_{w}\equiv G_{w}(\ii\eta)$ with $\eta=N^{-2}$.
	\end{lem}
	\begin{proof}
		The proof is a straight forward modification of Theorem 5.2 in \cite{Cipolloni-Erdos-Schroder2019arXiv}. The only difference is that we use \eqref{eq:beta_lb} and 
		\beq\label{eq:GGll_prf}
			\absv{\brkt{A(G_{z}BG_{w})}}\prec 1, \qquad \absv{\brkt{\bsx,G_{z}BG_{w}\bsy}}\prec 1,
		\eeq
		which is due to \eqref{eq:normG}. Using \eqref{eq:normG} and \eqref{eq:GGll_prf}, we may replace all $\eta_{*}$ appearing in the proof of Theorem 5.2 in \cite{Cipolloni-Erdos-Schroder2019arXiv} by $1$, for example in Eq. (5.16) therein. Also, we may skip the iteration of Eq. (5.14) in \cite{Cipolloni-Erdos-Schroder2019arXiv} since we can use \eqref{eq:GGll_prf} from the beginning. Finally, feeding \eqref{eq:beta_lb} into Eq. (5.9) in \cite{Cipolloni-Erdos-Schroder2019arXiv} proves the result.
	\end{proof}

	Lemma \ref{lem:2Gll} shows that $M_{B}$ can be used as a deterministic approximation to products of two resolvents with a deterministic matrix $B$ in between that arises along our calculations. First we record the small $\eta$ asymptotics of $M_{B}$. For $z,w\in\bbD_{1+\delta}^{c}$, we have
	\beqs
	\begin{split}
		\wh{\caB}(z,w)[M_{B}]=&\begin{pmatrix}
			M_{B}^{[11]}-m_{z}m_{w}\brkt{M_{B}^{[22]}}+z\ol{w}u_{z}u_{w}\brkt{M_{B}^{[11]}} &
			M_{B}^{[12]}-wm_{z}u_{w}\brkt{M_{B}^{[22]}}-zu_{z}m_{w}\brkt{M_{B}^{[11]}} \\
			M_{B}^{[21]}-\ol{z}u_{z}m_{w}\brkt{M_{B}^{[22]}}-\ol{w}m_{z}u_{w}\brkt{M_{B}^{[11]}} &
			M_{B}^{[22]}-\ol{z}wu_{z}u_{w}\brkt{M_{B}^{[22]}}-m_{z}m_{w}\brkt{M_{B}^{[11]}}
		\end{pmatrix},	\\
		\wh{\caB}(z,w)[M_{B}]=&M_{z}BM_{w},
		\end{split}
	\eeqs
	due to the definitions of $\wh{\caB}$ and $M_{B}$, where we abbreviated $M_{B}\equiv M_{B}(z,w)$. Taking the trace of each block, we find the following system of equations;
	\beqs
	\left\{
	\begin{array}{r}
		(1-z\ol{w}u_{z}u_{w})\brkt{M_{B}^{[11]}}-m_{z}m_{w}\brkt{M_{B}^{[22]}}=\brkt{(M_{z}BM_{w})^{[11]}},\\
		-m_{z}m_{w}\brkt{M_{B}^{[11]}}+(1-\ol{z}wu_{z}u_{w})\brkt{M_{B}^{[22]}}=\brkt{(M_{z}BM_{w})^{[22]}}.
		\end{array}\right.
	\eeqs
	Solving this equation gives
	\beqs
	\begin{split}
	\begin{pmatrix}
		\brkt{M_{B}^{[11]}}\\
		\brkt{M_{B}^{[22]}}
	\end{pmatrix}
		&=\left(\absv{1-u_{z}u_{w}z\ol{w}}^{2}-m_{z}^{2}m_{w}^{2}\right)^{-1}
		\begin{pmatrix}
			1-u_{z}u_{w}\ol{z}w & m_{z}m_{w} \\
			m_{z}m_{w} & 1-u_{z}u_{w}z\ol{w}
		\end{pmatrix}
	\begin{pmatrix}
		\brkt{(M_{z}BM_{w})^{[11]}} \\ \brkt{(M_{z}BM_{w})^{[22]}}
	\end{pmatrix}\\
	&=
	\begin{pmatrix}
		(1-\ol{z}^{-1}w^{-1})^{-1} & 0 \\ 0 & (1-z^{-1}\ol{w}^{-1})^{-1}
	\end{pmatrix}
	\begin{pmatrix}
		\brkt{ (M_{z}BM_{w})^{[11]}} \\ \brkt{(M_{z}BM_{w})^{[22]}}
	\end{pmatrix}
	+O(\eta^{2}),
	\end{split}
	\eeqs
	where we used \eqref{eq:asymp_M} in the last equality. In particular, plugging in $B=A^{\{kl\}}$, we have
	\beq\label{eq:2G}
	\begin{aligned}
		\begin{pmatrix}\brkt{M_{A^{\{11\}}}(z,w)^{[11]}}\\\brkt{M_{A^{\{11\}}}(z,w)^{[22]}}\end{pmatrix}
		&=\frac{\brkt{A}}{z\ol{w}-1}\begin{pmatrix} 0 \\ 1 \end{pmatrix}+O(\eta),&
		\begin{pmatrix}
			\brkt{M_{A^{\{12\}}}(z,w)^{[11]}} \\
			\brkt{M_{A^{\{12\}}}(z,w)^{[22]}}
		\end{pmatrix}
		&=O(\eta) 	\\
		\begin{pmatrix}
			\brkt{M_{A^{\{21\}}}(z,w)^{[11]}} \\
			\brkt{M_{A^{\{21\}}}(z,w)^{[22]}}
		\end{pmatrix}&=O(\eta), &
		\begin{pmatrix}
			\brkt{M_{A^{\{22\}}}(z,w)^{[11]}} \\
			\brkt{M_{A^{\{22\}}}(z,w)^{[22]}}
		\end{pmatrix}
		&=\frac{\brkt{A}}{\ol{z}w-1}\begin{pmatrix} 1 \\ 0 \end{pmatrix}+O(\eta).
	\end{aligned}
	\eeq
	Since $M_{B}(z,w)=M_{z}(B+\caS[M_{B}(z,w)])M_{w}$, these formulas directly imply asymptotics for $M_{B}(z,w)$ as follows.
	\beq\label{eq:2G_mat}
	\begin{aligned}
		M_{A^{\{11\}}}(z,w)&=\frac{1}{z\ol{w}}\left(A+\frac{\brkt{A}}{z\ol{w}-1}\right)^{\{22\}}+O_{\norm{\cdot}}(\eta),	&
		M_{A^{\{12\}}}(z,w)&=\frac{1}{zw}A^{\{21\}}+O_{\norm{\cdot}}(\eta),	\\
		M_{A^{\{21\}}}(z,w)&=\frac{1}{\ol{zw}}A^{\{12\}}+O_{\norm{\cdot}}(\eta),	&
		M_{A^{\{22\}}}(z,w)&=\frac{1}{\ol{z}w}\left(A+\frac{\brkt{A}}{\ol{z}w-1}\right)^{\{11\}}+O_{\norm{\cdot}}(\eta),
	\end{aligned}
	\eeq
	where the remainders have operator norms of size $O(\eta)$ due to \eqref{eq:asymp_M}.

	Another notation we recall from \cite{Cipolloni-Erdos-Schroder2019arXiv} is the definition of \emph{self-renormalization};
	\beq\label{eq:defn_renorm}
		\ul{Wf(W)}\deq Wf(W)-\E_{\wt{W}}\wt{W} (\partial_{\wt{W}}f)(W),
	\eeq
	where the function $f:M_{2N}(\C)\to M_{2N}(\C)$ has bounded derivative, $\wt{W}$ is the Hermitization of a complex Ginibre ensemble as in Definition \ref{defn:W}, and $(\partial_{\wt{W}}f)(W)$ denotes the directional derivative along $\wt{W}$ evaluated at $W$. Taking $f(W)$ to be $G_{z}$ gives $\ul{WG_{z}}=WG_{z}+\caS[G_{z}]G_{z}$, and from \eqref{eq:Dyson_M} we further have
	\beq\label{eq:renorm}
		G_{z}-M_{z}=-M_{z}\ul{WG_{z}}+M_{z}\caS[G_{z}-M_{z}]G_{z}.
	\eeq
	As we will see in the following sections, Lemma \ref{lem:ll} proves that the second term of \eqref{eq:renorm} is small since $\caS$ acts as a partial trace, allowing us to consider $\caS[G-M]$ as a constant of size $O_{\prec}(N^{-1})$ in effect.
\subsection{Proof of \eqref{eq:comput1}}\label{sec:prf_comput1}
	Finally we are ready to prove \eqref{eq:comput1}. Since this section concerns only three spectral parameters $z_{1},z_{i},$ and $w_{j}$, we abbreviate $G_{1}\equiv G_{z_{1}}$, $G_{i}\equiv G_{z_{i}},$ $G_{j}\equiv G_{w_{j}}$, and the same for $M_{z}$'s. We first apply \eqref{eq:renorm} to the factor of $G_{i}^{[22]}$ in $\frX_{i}$:
	\beq
	\begin{split}
		G_{i}^{[22]}G_{1}^{[21]}A_{1}G_{i}^{[11]}A_{i}
		=&M_{i}^{[22]}G_{1}^{[21]}A_{1}G_{i}^{[11]}A_{i}
		-(M_{i}\ul{WG_{i}})^{[22]}G_{1}^{[21]}A_{1}G_{i}^{[11]}A_{i}	\\
		&+(M_{i}\caS[G_{i}-M_{i}]G_{i})^{[22]}G_{1}^{[21]}A_{1}G_{i}^{[11]}A_{i}	\nonumber\\
		=&(M_{i}I^{\{22\}}G_{1})^{[21]}A_{1}G_{i}^{[11]}A_{i}
		+(M_{i}\caS[G_{i}I^{\{22\}}G_{1}]G_{1})^{[21]}A_{1}G_{i}^{[11]}A_{i}	\\
		&+(M_{i}\caS[G_{i}I^{\{22\}}G_{1}A_{1}^{\{11\}}G_{i}]G_{i})^{[21]}A_{i}	
		-(M_{i}\ul{WG_{i}I^{\{22\}}G_{1}A_{1}^{\{11\}}G_{i}})^{[21]}A_{i}	\label{eq:X1}\\
		&+(M_{i}\caS[G_{i}-M_{i}]G_{i})^{[22]}G_{1}^{[21]}A_{1}G_{i}^{[11]}A_{i}.	\nonumber
	\end{split}
	\eeq
	Note that in the second equality we applied \eqref{eq:defn_renorm} twice, namely
	\beqs
	\begin{split}
		\ul{WG_{i}I^{\{22\}}G_{1}}&=\ul{WG_{i}}I^{\{22\}}G_{1}
		+\caS[G_{i}I^{\{22\}}G_{1}]G_{1},\\
		\ul{WG_{i}I^{\{22\}}G_{1}A_{1}^{\{11\}}G_{i}}
		&=\ul{WG_{i}I^{\{22\}}G_{1}}A_{1}^{\{11\}}G_{i}+\caS[G_{i}I^{\{22\}}G_{1}A_{1}^{\{11\}}G_{i}]G_{i}.
	\end{split}
	\eeqs
	By Lemma \ref{lem:2Gll}, \eqref{eq:2G}, and \eqref{eq:2G_mat}, the contribution of the first two terms on the right-hand side of \eqref{eq:X1} is given by
	\begin{align}
		&\brkt{(M_{i}(I^{\{22\}}+\caS[G_{i}I^{\{22\}}G_{1}])G_{1})^{[21]}A_{1}G_{i}^{[11]}A_{i}}	\nonumber\\
		=&\brkt{(M_{i}(I^{\{22\}}+\caS[M_{I^{\{22\}}}(z_{i},z_{1})])M_{A_{1}^{\{11\}}}(z_{1},z_{i}))^{[21]}A_{i}}+O_{\prec}(N^{-1})	
		=O(\eta)+O_{\prec}(N^{-1})=O_{\prec}(N^{-1}),	\label{eq:X1_12}
	\end{align}
	where the last equality is due to our choice of $\eta=N^{-2}$.
	
	Before we proceed to the third term, we consider the last two terms of \eqref{eq:X1}. The fourth term of \eqref{eq:X1} can be estimated using the exact same argument as in \cite[equation (6.21)]{Cipolloni-Erdos-Schroder2019arXiv}, except we use the norm bound \eqref{eq:normG} to replace all factors of $\eta^{-1}$ by $1$. As a result, for generic deterministic matrices $\bsA,\bsA',\bsA''\in M_{2N}(\C)$ with bounded operator norms we find that
	\beq\label{eq:3G_var}
		\Expct{\Big\vert{\frac{1}{N}\Tr{\ul{WG_{i}\bsA G_{1}\bsA'G_{i}\bsA''}}\Big\vert^{2}}}=O(N^{-2}).
	\eeq
	The normalized trace of the fourth term of \eqref{eq:X1} has exactly the same form with $\bsA=I^{\{22\}},\bsA'=A_{1}^{\{11\}}$, and $\bsA''=A_{i}^{\{12\}}M_{i}$, so that its second moment is bounded by $N^{-2}$. For the last term in \eqref{eq:X1}, we use Lemma \ref{lem:ll} to each block of $\caS[G-M]$ and apply the norm bound \eqref{eq:normG} to the rest, which leads to
	\beq\label{eq:X1_5}
		\absv{\brkt{(M_{i}\caS[G_{i}-M_{i}]G_{i})^{[22]}G_{1}^{[21]}A_{1}G_{i}^{[11]}A_{i}}}\prec \frac{1}{N}.
	\eeq
	
	In order to deal with the third term of \eqref{eq:X1}, we repeat the same expansion as \eqref{eq:X1} without the last factor of $A_{i}$. To be precise, we write
	\beq\label{eq:X11}
	\begin{split}
		G_{i}I^{\{22\}}G_{1}A_{1}^{\{11\}}G_{i}&=M_{i}I^{\{22\}}G_{1}A_{1}^{\{11\}}G_{i}+M_{i}\caS[G_{i}I^{\{22\}}G_{1}]G_{1}A_{1}^{\{11\}}G_{i}+M_{i}\caS[G_{i}I^{\{22\}}G_{1}A_{1}^{\{11\}}G_{i}]G_{i}\\
		&-M_{i}\ul{WG_{i}I^{\{22\}}G_{1}A_{1}^{\{11\}}G_{i}}+ M_{i}\caS[G_{i}-M_{i}]G_{i}I^{\{22\}}G_{1}A_{1}^{\{11\}}G_{i}.
	\end{split}
	\eeq
	Then we take the partial trace over the $(1,1)$-th block of \eqref{eq:X11} to get 
	\beq\label{eq:X111}
	\begin{split}
		\brkt{(G_{i}I^{\{22\}}&G_{1}A_{1}^{\{11\}}G_{i})^{[11]}} 
		=\brkt{(M_{i}(I^{\{22\}}+\caS[M_{I^{\{22\}}}(z_{i},z_{1})])M_{A_{1}^{\{11\}}}(z_{1},z_{i}))^{[11]}}	\\
		&+\brkt{(M_{i}\caS[G_{i}I^{\{22\}}G_{1}A_{1}^{\{11\}}G_{i}]M_{i})^{[11]}}
		-\brkt{(M_{i}\ul{WG_{i}I^{\{22\}}G_{1}A_{1}^{\{11\}}G_{i}})^{[11]}}+O_{\prec}(N^{-1})	\\
		=&\frac{z_{1}\ol{z}_{i}}{\absv{z_{1}\ol{z}_{i}-1}^{2}}m_{i}\brkt{A_{1}}
		+\absv{z_{i}}^{-2}\brkt{(G_{i}I^{\{22\}}G_{1}A_{1}^{\{11\}}G_{i})^{[22]}}-e_{i1}+O_{\prec}(N^{-1})	\\
		=&\absv{z_{i}}^{-2}\brkt{(G_{i}I^{\{22\}}G_{1}A_{1}^{\{11\}}G_{i})^{[22]}}-e_{i1}+O_{\prec}(N^{-1}),
	\end{split}
	\eeq
	where we used respectively Lemma \ref{lem:2Gll} and \eqref{eq:2G_mat} in the first and second equalities, and abbreviated 
	\beqs
		e_{i1}\deq \brkt{(M_{i}\ul{WG_{i}I^{\{22\}}G_{1}A_{1}^{\{11\}}G_{i}})^{[11]}}.
	\eeqs
	By the same reasoning, taking the partial trace of the $(2,2)$-th block of $\eqref{eq:X11}$ gives
	\beq\label{eq:X112}
		\brkt{(G_{i}I^{\{22\}}G_{1}A_{1}^{\{11\}}G_{i})^{[22]}}=\absv{z_{i}}^{-2}\brkt{(G_{i}I^{\{22\}}G_{1}A_{1}^{\{11\}}G_{i})^{[11]}}-e_{i2}+O_{\prec}(N^{-1}),
	\eeq
	where $e_{i2}$ is the $(2,2)$-th partial trace of the same $(2N\times 2N)$ matrix in the definition of $e_{i1}$. Solving \eqref{eq:X111} and \eqref{eq:X112} with respect to the partial traces leads to
	\beq\label{eq:X11_conc}
		\norm{\caS[G_{i}I^{\{22\}}G_{1}A_{1}^{\{11\}}G_{i}]}\leq \absv{\brkt{(G_{i}I^{\{22\}}G_{1}A_{1}^{\{11\}}G_{i})^{[11]}}}+\absv{\brkt{(G_{i}I^{\{22\}}G_{1}A_{1}^{\{11\}}G_{i})^{[22]}}}
		\lesssim \absv{e_{i1}}+\absv{e_{i2}}+O_{\prec}(N^{-1}).
	\eeq
	Note that $e_{i1}$ and $e_{i2}$ both have the same form as \eqref{eq:3G_var}, so that $\bbE[\absv{e_{i1}}^{2}],\bbE[\absv{e_{i2}}^{2}]$ are $O(N^{-2})$. Therefore the normalized trace of the third term on the right-hand side of \eqref{eq:X1} has second moment of $O(N^{-2+\epsilon})$.
	
	Plugging \eqref{eq:X1_12}, \eqref{eq:3G_var}, \eqref{eq:X1_5}, and \eqref{eq:X11_conc} into \eqref{eq:X1}, we obtain that
	\beq\label{eq:X1_conc}
	\frX_{i}=R_{i}+O_{\prec}(N^{-1})
	\eeq
	where $R_{i}$ is a random variable with $\expct{\absv{R_{i}}^{2}}=O(N^{-2})$ uniformly over $i$. This completes the proof of the first estimate in \eqref{eq:comput1}.

	The proof of the second estimate in \eqref{eq:comput1} follows similar lines, and the only difference is that the analogue of \eqref{eq:X1_12} no longer has negligible contribution. To be precise, the exact same reasoning as above proves
	\beq\label{eq:Y11}
		\frY_{j}=\brkt{G_{j}^{[12]}G_{1}^{[21]}A_{1}G_{j}^{[12]}B_{j}\adj}	=\brkt{(M_{j}(I^{\{22\}}+\caS[G_{j} I^{\{22\}}G_{1}])G_{1})^{[11]}A_{1}G_{j}^{[12]}B_{j}\adj} +R_{j}+O_{\prec}(N^{-1}),
	\eeq
	where $R_{j}$ satisfies $\bbE[\absv{R_{j}}^{2}]=O(N^{-2})$. To avoid repetition, we omit the detailed proof of \eqref{eq:Y11} and just mention the key changes. The first term on the right-hand side of \eqref{eq:Y11}, in contrast to \eqref{eq:X1_12}, concentrates around a deterministic value of constant order. This value can be calculated using $\eta=N^{-2}$, Lemma \ref{lem:2Gll}, and \eqref{eq:2G_mat} as
	\beq\label{eq:Y12}
	\begin{split}
		\brkt{(M_{j}(I^{\{22\}}+\caS[G_{j} I^{\{22\}}&G_{1}])G_{1})^{[11]}A_{1}G_{j}^{[12]}B_{j}\adj}	\\
		=&\brkt{(M_{j}(I^{\{22\}}+\caS[M_{I^{\{22\}}}(w_{j},z_{1})])M_{A_{1}^{\{11\}}}(z_{1},w_{j}))^{[12]}B_{j}\adj}+O_{\prec}(N^{-1})	\\
		=&\frac{z_{1}}{1-\ol{w}_{j}z_{1}}\left(\frac{\brkt{A_{1}}\brkt{B_{j}\adj}}{1-z_{1}\ol{w}_{j}}+\brkt{M_{z_{1}}^{[21]}A_{1}M_{w_{j}}^{[12]}B_{j}\adj}\right)+O(\eta)+O_{\prec}(N^{-1})	\\
		=&z_{1}\frac{\brkt{A_{1}}\brkt{B_{j}\adj}}{(1-z_{1}\ol{w}_{j})^{2}}+\frac{\brkt{A_{1}B_{j}\adj}}{\ol{w}_{j}(1-z_{1}\ol{w}_{j})}+O_{\prec}(N^{-1}).
	\end{split}
	\eeq
	Combining \eqref{eq:Y11} and \eqref{eq:Y12} proves the second estimate in \eqref{eq:comput1}.
	
	\subsection{Proof of \eqref{eq:comput2}}\label{sec:prf_comput2}
	Prior to the proof of \eqref{eq:comput2}, we remark that Lemma \ref{lem:ll}, $\eta=N^{-2}$, and \eqref{eq:asymp_M} imply the following;
	\beq\label{eq:entry_prior}
		G^{[kl]}_{ab}\prec \delta_{ab}\lone_{k\neq l}+N^{-1/2}.
	\eeq
	In other words, $G^{[kl]}_{ab}$ is $O_{\prec}(N^{-1/2})$ unless it is a diagonal entry of either $G^{[12]}$ or $G^{[21]}$. This a priori bound will be repeatedly used throughout the remaining sections, since the calculations largely involve case-by-case estimates of resolvent entries.
	
	We now turn to the proof of \eqref{eq:comput2}, whose left-hand side can be expanded as
	\beq\label{eq:higher}
	\begin{split}
		&\frac{\kappa(k,l+1)}{k!l!}\sum_{a,b}\partial^{(k,l)}_{ab}[(G_{1}^{[21]}A_{1})_{ba}\ctr{\caX^{(1)}\caY}]		\\
		=&\sum_{a,b}\kappa(k,l+1)\sum_{k_{1}\leq k,l_{1}\leq l}\left(\frac{\partial^{(k_{1},l_{1})}_{ab}[(G_{1}^{[21]}A_{1})_{ba}]}{k_{1}!l_{1}!}\right)\left(\frac{\partial^{(k-k_{1},l-l_{1})}_{ab}[\ctr{\caX^{(1)}\caY}]}{(k-k_{1})!(l-l_{1})!}\right).
	\end{split}
	\eeq
	When $k+l\geq 4$, by \eqref{eq:normG} we have for all $k_{1}\leq k$ and $l_{1}\leq l$ that
	\beqs
		\partial_{ab}^{(k_{1},l_{1})}[(G_{1}^{[21]}A_{1})_{ba}]\prec 1, \qquad \partial_{ab}^{(k-k_{1},l-l_{1})}[\ctr{\caX^{(1)}\caY}]\prec 1.
	\eeqs
	Thus $\kappa(k,l+1)=O(N^{-(k+l+1)/2})$ gives
	\beqs
		\Absv{\sum_{a,b}\frac{\kappa(k,l+1)}{k!l!}\partial^{(k,l)}_{ab}[(G_{1}^{[21]}A_{1})_{ba}\ctr{\caX^{(1)}\caY}]}
		\prec N^{2-(k+l+1)/2},
	\eeqs
	which exactly matches \eqref{eq:comput2}. Therefore we assume $k+l\in\{2,3\}$ in what follows.
	
	We first consider the case when $k_{1}+l_{1}=k+l$, that is, when all the differentials operators in \eqref{eq:higher} act on the first factor $(G_{1}^{[21]}A_{1})_{ba}$. Then the derivative consists of terms of the form
	\beq\label{eq:der_GA}
		\Tr G_{1}\Delta_{1}G_{1}\Delta_{2}\cdots\Delta_{k+l}G_{1}(A_{1}\Delta_{ab})^{\{12\}},
	\eeq
	where $\Delta_{i}$'s are either $\Delta_{ab}^{\{12\}}$ or $\Delta_{ba}^{\{21\}}$, which reduces to products of $(k+l)$ entries of $G$ and an entry of $GA^{\{12\}}$. In particular, for any choice of $\Delta_{i}$, each entry of $G$ appearing in \eqref{eq:der_GA} must be one of the following;
	\beqs
		G^{[11]}_{aa}, \quad G^{[12]}_{ab}, \quad G^{[21]}_{ba}, \quad G^{[22]}_{bb}.
	\eeqs
	Hence \eqref{eq:entry_prior} proves that
	\beqs
		\absv{\Tr G_{1}\Delta_{1}\cdots\Delta_{k+l}G_{1}(A_{1}\Delta_{ab})^{\{12\}}}\prec N^{-(k+l)/2}+\delta_{ab},
	\eeqs
	so that 
	\beq\label{eq:der_XX}
		\kappa(k,l+1)\sum_{a,b}\ctr{\partial^{(k,l)}_{ab}[\Tr G_{1}^{\{21\}}A_{1}\Delta_{ab}]}=O_{\prec}(N^{3/2-(k+l)})+O_{\prec}(N^{1/2-(k+l)/2})=O_{\prec}(N^{-1/2}).
	\eeq
	Therefore we have that
	\beqs
		\kappa(k,l+1)\sum_{a,b}\Expct{\ctr{\caX^{(1)}\caY}\partial^{(k,l)}_{ab}[(G_{1}^{[21]}A_{1})_{ba}]}
		=O(N^{-1/2+\epsilon}).
	\eeqs
	Thus it suffices to consider the case where $k+l=2,3$ and $k_{1}+l_{1}<k+l$, that is, we need to estimate the following quantity;
	\beq\label{eq:remainder}
		\sum_{a,b}\sum_{\bsk,\bsl}\kappa(\absv{\bsk},\absv{\bsl}+1)\frac{\partial^{(k_{1},l_{1})}_{ab}[(G_{1}^{[21]}A_{1})_{ba}]}{k_{1}!l_{1}!}\prod_{i\in\bbrktt{2,p}}\frac{\partial^{(k_{i},l_{i})}_{ab}[\caX_{i}]}{k_{i}!l_{i}!}\prod_{j\in\bbrktt{q}}\frac{\partial^{(k_{p+j},l_{p+j})}_{ab}[\caY_{j}]}{k_{p+j}!l_{p+j}!},
	\eeq
	where $\bsk$ and $\bsl$ run over $(p+q)$-tuples of nonnegative integers with $\absv{\bsk}+\absv{\bsl}=2$ or $3$ and $k_{1}+l_{1}<\absv{\bsk}+\absv{\bsl}$, and we denoted $\absv{\bsk}=\sum_{i\in\bbrktt{p+q}}k_{i}$ and $\absv{\bsl}=\sum_{i\in\bbrktt{p+q}}l_{i}$. Here we used the fact that at least one differential operator acts on $\ctr{\caX^{(1)}\caY}$, so that its derivative is equal to that of the uncentered quantity $\caX^{(1)}\caY$.
	
	In what follows, we present an upper bound of the contribution of the sum in \eqref{eq:remainder} for $\absv{\bsk}+\absv{\bsl}=2$ case by case. The case $\absv{\bsk}+\absv{\bsl}=3$ will be handled in a similar fashion afterwards.
	\begin{itemize}
		\item[\textbf{Case 1:}] $k_{i}=2$ or $l_{i}=2$ for some $i>1$.\\
		Due to similarity, we consider only the case $k_{i}=2$ for some $i\in\bbrktt{2,p}$. In this case, the whole contribution is
		\beq\label{eq:ki=2}
			\frac{\kappa(2,1)}{2}\caX^{(1,i)}\caY\sum_{a,b}(G_{1}^{[21]}A_{1})_{ba}\partial_{ab}^{(2,0)}[\Tr G_{i}^{[21]}A_{i}]
			=\kappa(2,1)\caX^{(1,i)}\caY\sum_{a,b}(G_{1}^{[21]}A_{1})_{ba}(\Tr{G_{i}^{[21]}\Delta_{ab}G_{i}^{[21]}\Delta_{ab}G_{i}^{[21]}A_{i}}).
		\eeq
		Note that the first factor on the right-hand side of \eqref{eq:ki=2} is given by
		\beq\label{eq:factor1}
			(G_{1}^{[21]}A_{1})_{ba}=(M_{1}^{[21]}A_{1})_{ba}+O_{\prec}(N^{-1/2})=-\frac{1}{z_{1}}(A_{1})_{ba}+O_{\prec}(N^{-1/2}),
		\eeq
		using Lemma \ref{lem:ll}, and second factor admits the following bound;
		\beqs
			\absv{\Tr{G_{i}^{[21]}\Delta_{ab}G_{i}^{[21]}\Delta_{ab}G_{i}^{[21]}A_{i}}}=\absv{(G_{i}^{[21]})_{ba}(G_{i}^{[21]}A_{i}G_{i}^{[21]})_{ba}}\prec N^{-1/2}\absv{(A_{i})_{ba}}+N^{-1},
		\eeqs
		where we applied Lemma \ref{lem:2Gll} and \eqref{eq:2G_mat} to get
		\beqs
			(G_{i}^{[21]}A_{i}G_{i}^{[21]})_{ba}=(M_{i}A_{i}^{\{12\}}M_{i})^{[21]}_{ba}+O_{\prec}(N^{-1/2})=\absv{z_{i}}^{-2}(A_{i})_{ba}+O_{\prec}(N^{-1/2}).
		\eeqs
		Combining with $\kappa(2,1)=O(N^{-3/2})$ and \eqref{eq:prior_XY}, the left-hand side of \eqref{eq:ki=2} is stochastically dominated by
		\beqs
			N^{-3/2}\sum_{a,b}\absv{(A_{1})_{ba}}\absv{(A_{i})_{ba}}+N^{-1/2}\leq N^{-3/2}(\Tr \absv{A_{1}}^{2}+\Tr \absv{A_{2}}^{2})+N^{-1/2}\prec N^{-1/2}(1+\norm{A_{1}}+\norm{A_{i}}),
		\eeqs
		which establishes \eqref{eq:comput2} for this case. The same argument applies for $i>p$ and for $l_{i}=2$.
		
		\item[\textbf{Case 2:}] $k_{i}=1=l_{i}$ for some $i>1$. \\
		As above, we focus on the case $i\in\bbrktt{2,p}$. In this case the contribution becomes
		\beqs
		\begin{split}
			&\kappa(1,2)\caX^{(1,i)}\caY\sum_{a,b}(G_{1}^{[21]}A_{1})_{ba} 
			\Tr (G_{i}^{[22]}\Delta_{ba}G_{i}^{[11]}\Delta_{ab} G_{i}^{[21]}+G_{i}^{[21]}\Delta_{ab}G_{i}^{[22]}\Delta_{ba}G_{i}^{[11]})A_{i}	\\
			=&\kappa(1,2)\caX^{(1,i)}\caY\sum_{a,b}(G_{1}^{[21]}A_{1})_{ba} \left((G_{i}^{[11]})_{aa}(G_{i}^{[21]}A_{i}G_{i}^{[22]})_{bb}+(G_{i}^{[22]})_{bb}(G_{i}^{[11]}A_{i}G_{i}^{[21]})_{bb}\right).
		\end{split}
		\eeqs
		Then we use Lemma \ref{lem:2Gll} and \eqref{eq:2G_mat} to get
		\beqs
		\begin{aligned}
			&(G_{i}^{[21]}A_{i}G_{i}^{[22]})_{bb}=(M_{A_{i}^{\{12\}}}(z_{i},z_{i})^{[22]})_{bb}+O_{\prec}(N^{-1/2})=O_{\prec}(N^{-1/2}),	\\
			&(G_{i}^{[11]}A_{i}G_{i}^{[21]})_{bb}=(M_{A_{i}^{\{12\}}}(z_{i},z_{i})^{[11]})_{bb}+O_{\prec}(N^{-1/2})=O_{\prec}(N^{-1/2}).
		\end{aligned}
		\eeqs
		By \eqref{eq:entry_prior}, \eqref{eq:factor1}, and $\kappa(1,2)=O(N^{-3/2})$, the contribution of this case is stochastically dominated by 
		\beq\label{eq:CS_A}
			N^{-5/2}\sum_{a,b}\absv{(A_{1})_{ba}}\leq N^{-3/2}(\Tr\absv{A_{1}}^{2})^{1/2}\leq \norm{A_{1}}N^{-1}.
		\eeq
		
		\item[\textbf{Case 3:}] $k_{1}+l_{1}=1$.\\
		In this case, one derivative is acting on the first factor $\Tr (G_{1}^{[21]}A_{1}\Delta_{ab})$ in \eqref{eq:remainder}, resulting in a factor of
		\beq\label{eq:der_1}
			\partial_{ab}^{(k_{1},l_{1})}[\Tr G_{1}^{[21]}A_{1}\Delta_{ab}]=
			\begin{cases}
				(G_{1}^{[21]})_{ba}(G^{[21]}A_{1})_{ba} & \text{if }k_{1}=1,\\
				(G_{1}^{[22]})_{bb}(G^{[11]}A_{1})_{aa} & \text{if }l_{1}=1.
			\end{cases}
		\eeq
		By \eqref{eq:entry_prior}, this factor is $O_{\prec}(N^{-1/2})$. Then the remaining one differential should act on either $\caX_{i}$ or $\caY_{j}$. If it acts on $\caX_{i}$, we get a factor of
		\beq\label{eq:der_X}
		\partial_{ab}^{(k_{i},l_{i})}\caX_{i}=\begin{cases}
			(G_{i}^{[21]}A_{i}G_{i}^{[21]})_{ba}\prec \absv{(A_{i})_{ba}}+N^{-1/2} & \text{if }k_{i}=1,\\
			(G_{i}^{[11]}A_{i}G_{i}^{[22]})_{ba}\prec N^{-1/2} & \text{if }l_{i}=1.
		\end{cases}
		\eeq
		Applying Cauchy-Schwarz inequality to $\absv{(A_{i})_{ba}}$ as in \eqref{eq:CS_A} proves that the total contribution is of order $O_{\prec}(N^{-1/2})$, and the same applies when the remaining differential acts on $\caY_{j}$'s.
		
		\item[\textbf{Case 4:}] $k_{i}+l_{i}=1=k_{j}+l_{j}$ for some $i\neq j>1$.	\\
		By \eqref{eq:der_X} and Cauchy-Schwarz inequality, the contribution of this case is dominated by
		\beqs
			N^{-3/2}(\Tr\absv{A_{i}}^{2}+\Tr\absv{B_{j}}^{2})+N^{-2}\sum_{a,b}(\absv{(A_{i})_{ab}}+\absv{(B_{j})_{ab}})+N^{-1/2}=O(N^{-1/2})
		\eeqs
		if $i\in\bbrktt{2,p}$ and $j\in\bbrktt{p+1,p+q}$. The same reasoning applies to other choices of $i$ and $j$, proving \eqref{eq:comput2}.
	\end{itemize}
	The four cases above exhaust all possible choices of $(\bsk,\bsl)$ with $\absv{\bsk}+\absv{\bsl}=2$.
	
	Since the proof of \eqref{eq:comput2} for $\absv{\bsk}+\absv{\bsl}=3$ is almost the same, we only present its sketch. Since $\kappa(\absv{\bsk},\absv{\bsl}+1)=O(N^{-2})$, the sum $\kappa(\absv{\bsk},\absv{\bsl}+1)\sum_{a,b}$ becomes the average over $a$ and $b$, and we only need this average to be $O_{\prec}(N^{-1/2})$. For example, using \eqref{eq:factor1} and \eqref{eq:CS_A}, the contribution when $k_{1}+l_{1}=0$ is bounded by
	\beqs
		N^{-2+\epsilon}\sum_{a,b}\absv{(A_{1})_{ba}}+N^{-1/2+\epsilon}\leq N^{-1/2+\epsilon}(1+\norm{A_{1}}),
	\eeqs
	where we simply applied \eqref{eq:prior_XY} and the norm bounds $\norm{G_{z}},\norm{A_{i}},\norm{B_{j}}=O(1)$ to factors other than $\Tr G_{1}^{[21]}A_{1}\Delta_{ab}$. By the same reasoning, \eqref{eq:der_1} and \eqref{eq:der_X} exhaust the case when $k_{i}+l_{i}=1$ for some $i\in\bbrktt{p+q}$. Since the only remaining case is $k_{1}+l_{1}=3=\absv{\bsk}+\absv{\bsl}$, which was already dealt with in \eqref{eq:der_XX}, this completes the proof of \eqref{eq:comput2}.
	
	\subsection{Proof of \eqref{eq:mean_asymp_G}}\label{sec:mean_prf}
	In this section, we prove the asymptotics \eqref{eq:mean_asymp_G} of $\expct{\Tr G^{[21]}A}$. First, we repeat the same expansion using $\ii\eta G+I=W_{z}G$ as in \eqref{eq:Stein_2};
	\beq\label{eq:mean_expa}
	\begin{split}
		\expct{\ii\eta&\Tr G^{[11]}_{1}A_{1}}=\expct{\Tr XG^{[21]}_{1}A_{1}}-z_{1}\expct{\Tr G^{[21]}_{1}A_{1}}-\Tr A_{1}	\\
		=&-z_{1}\expct{\Tr G^{[21]}_{1}A_{1}}-\Tr A_{1}
		+\sum_{\substack{k,l \in\Z_{+}}}\frac{\kappa(k+1,l)}{k!l!}\sum_{a,b}\expct{\partial_{ab}^{(k,l)}[\Tr \Delta_{ab}G^{[21]}_{1}A_{1}]}.
	\end{split}
	\eeq
	Therefore, in order to prove \eqref{eq:mean_asymp_G}, it suffices to prove that the left-hand side and the last term on the right-hand side of \eqref{eq:mean_expa} are $O(N^{-1/2+\epsilon})$. Since the left-hand side is $O_{\prec}(\eta)$ by \eqref{eq:prior}, we may focus on the latter.
	
	Note that $\kappa(2,0)=0$, so that the term corresponding to $(k,l)=(1,0)$ vanishes. When $k=0$ and $l=1$, we have
	\beqs
		\kappa(1,1)\sum_{a,b}\partial_{ab}^{(0,1)}[\Tr \Delta_{ab}G^{[21]}_{1}A_{1}]=-\brkt{G^{[22]}_{1}}\Tr G^{[11]}_{1}A_{1}\prec N(N^{-1}+\eta)^{2}\prec N^{-1},
	\eeqs
	where we used Lemma \ref{lem:ll} and $\eta=N^{-2}$. Next, we trivially handle the case $k+l\geq 4$ using $\kappa(k+1,l)=O(N^{-(k+l+1)/2})$ as in the previous section, that is,
	\beqs
		\kappa(k+1,l)\sum_{a,b}\partial_{ab}^{(k,l)}[\Tr \Delta_{ab}G^{[21]}_{1}A_{1}]\prec N^{-(k+l-3)/2} \qquad \text{if }k+l\geq 4.
	\eeqs
	Thus we have reduced the case to $k+l\in\{2,3\}$ and what we need to estimate is identical to the left-hand side of \eqref{eq:der_XX} except the centering. Therefore it suffices to recall that \eqref{eq:der_XX} followed from the bound without centering and that $\kappa(k,l+1)$ and $\kappa(k+1,l)$ are both $O(N^{-(k+l+1)/2})$. This completes the proof of \eqref{eq:mean_asymp_G}, hence that of Proposition \ref{prop:CLT_G} for complex $\chi$.

	\section{ Extension to real random matrices}\label{sec:real}
		This section is devoted to the proof of Proposition \ref{prop:CLT_G} when $\chi$ is real-valued. Since the proof is almost identical to the complex case, we only point out the necessary modifications. We will repeatedly use that $(G_{z}^{[12]})^{\intercal}=G_{\ol{z}}^{[21]}$ and $(G_{z}^{[11]})^{\intercal}=G_{\ol{z}}^{[11]}$ for real-valued $X$.
		
		First of all, the asymptotic Wick formula \eqref{eq:Wick_result} in Proposition \ref{prop:Wick} should be modified to
		\beq\label{eq:Wick_real}
			\expct{\caX\caY}=\sum_{\substack{P\in\mathrm{PairPart}(S,T) \\ Q\in\mathrm{PairPart}(S^{c},T^{c})}}\prod_{(k,\ell)\in P} V^{\circ}_{k,\ell}\prod_{(k',\ell')\in Q}V_{k',\ell'}+O(N^{-1/2+\epsilon}),
		\eeq
		where $\mathrm{PairPart}(I,J)$ for $I\subset\bbrktt{p},J\subset\bbrktt{q}$ denotes the set of partitions of index set $I\cup(J+p)\subset\bbrktt{p+q}$ into pairs and $V^{\circ}_{k,\ell}$ and $V_{k',\ell'}$ are defined by
		\begin{align*}
			V^{\circ}_{i,i'}&\deq V^{\circ}(z_{i},\ol{z}_{i'})\brkt{A_{i}A_{j}^{\intercal}}, &
			V_{i,i'}&\deq V(z_{i},\ol{z}_{i'}), \\
			V^{\circ}_{i,j+p}&\deq V^{\circ}(z_{i},w_{j})\brkt{A_{i}B_{j}\adj}, &
			V_{i,j+p}&\deq V(z_{i},w_{j}),	\\
			V^{\circ}_{j+p,j'+p}&\deq V^{\circ}(\ol{w}_{j},w_{j'})\ol{\brkt{B_{i}B_{j}^{\intercal}}}, &
			V_{j+p,j'+p}&\deq V(\ol{w}_{j},\ol{w}_{j'})
		\end{align*}
		for $i,i'\in \bbrktt{p}$ and $j,j'\in \bbrktt{q}$ with $V^{\circ}$ and $V$ defined in \eqref{eq:defn_V}. The second necessary modification is in \eqref{eq:Stein_2}; when $\chi$ is a real random variable, we have to replace the complex cumulant expansion in \eqref{eq:Stein_2} with the following real version;
		\beqs
		\expct{X_{ab} f(X_{ab})}=\sum_{k\in\N}\frac{\kappa(k+1)}{k!}\expct{\partial_{ab}^{k}[f](X_{ab})},
		\eeqs
		where $\kappa(k)$ and $\partial_{ab}^{k}$ are defined by the exact same formulas as $\kappa(k,0)$ and $\partial_{ab}^{(k,0)}$, respectively. The derivative of the resolvent is given by
		$\partial_{ab}[G_{z}^{[kl]}]=-G_{z}^{[k1]}\Delta_{ab}G_{z}^{[2l]}-G_{z}^{[k2]}\Delta_{ba}G_{z}^{[1l]}$. After applying these changes, the first term of \eqref{eq:Stein_2} becomes the expectation of
		\beq\label{eq:real_Stein_1}
		-\left({\brkt{G_{z_{1}}^{[22]}}\Tr G_{z_{1}}^{[11]}A_{1}}+\frac{1}{N}\sum_{a,b}(G_{z_{1}}^{[21]})_{ba}(G_{z_{1}}^{[21]}A)_{ba}\right)\ctr{\caX^{(1)}\caY},
		\eeq
		where the second term in the bracket was absent in the complex case. This term is written as
		\beqs
		\frac{1}{N}\sum_{a,b}(G_{z_{1}}^{[21]})_{ba}(G_{z_{1}}^{[21]}A)_{ba}=\brkt{G_{z_{1}}^{[21]}AG_{\ol{z}_{1}}^{[12]}}=\brkt{(G_{z_{1}}A^{\{11\}}G_{\ol{z}_{1}})^{[22]}}.
		\eeqs
		We will see below that Lemma \ref{lem:2Gll} and \eqref{eq:2G} still hold for real $X$, so that the  bound 
		 in \eqref{eq:Stein_1_1} remains true for the analogue of the first term of \eqref{eq:Stein_2} after an application of $\expct{\ctr{X}Y}=\expct{X\ctr{Y}}$. Similarly, \eqref{eq:Wick_start} holds if we replace $\frX_{i}$ and $\frY_{j}$ respectively by 
\beq\label{eq:real_XY}
		\begin{split}
			\wt{\frX}_{i}&\deq-\frac{1}{N}\sum_{a,b}(G_{z_{1}}^{[21]}A_{1})_{ba}\partial_{ab}[\Tr G_{z_{i}}^{[21]}A_{i}]
			=\frac{1}{N}\sum_{a,b}(G_{z_{1}}^{[21]}A_{1})_{ba}\Tr (G_{z_{i}}(\Delta_{ab}^{\{12\}}+\Delta_{ba}^{\{21\}})G_{z_{i}})^{[21]}A_{i}\\
			&=\frac{1}{N}\sum_{a,b}(G_{z_{1}}^{[21]}A_{1})_{ba}\left((G_{z_{i}}^{[21]}A_{i}G_{z_{i}}^{[21]})_{ba}+(G_{z_{i}}^{[11]}A_{i}G_{z_{i}}^{[22]})_{ab}\right)	
			=\frX_{i}+\brkt{G_{z_{1}}^{[21]}A_{1}G_{\ol{z}_{i}}^{[12]}A_{i}^{\intercal}G_{\ol{z}_{i}}^{[12]}},\\
			\wt{\frY}_{j}&\deq-N^{-1}\sum_{a,b}(G_{z_{1}}^{[21]}A_{1})_{ba}\partial_{ab}\Tr [G_{w_{j}}^{[12]}B_{j}\adj]
			=\frY_{j}+\brkt{G_{z_{1}}^{[21]}A_{1}G_{\ol{w}_{j}}^{[11]}\ol{B}_{j}G_{\ol{w}_{j}}^{[22]}}.
		\end{split}
		\eeq
		and we change remainders in cumulant expansions accordingly.
		
		Therefore, it suffices to show the following analogues of \eqref{eq:comput1} and \eqref{eq:comput2} in order to prove \eqref{eq:Wick_real};
		\beq\label{eq:comput_real}
		\begin{split}
			\Expct{\Absv{\wt{\frX}_{i}+\frac{\brkt{A_{1}A_{i}^{\intercal}}}{z_{i}(z_{1}z_{i}-1)}+\frac{z_{1}\brkt{A_{1}}\brkt{A_{i}}}{(z_{1}z_{i}-1)^{2}}}^{2}}&=O(N^{-2+\epsilon}),\\
			\Expct{\Absv{\wt{\frY}_{j}+\frac{\brkt{A_{1}B_{j}\adj}}{\ol{w}_{j}(z_{1}\ol{w}_{j}-1)}+\frac{z_{1}\brkt{A_{1}}\brkt{B_{j}\adj}}{(z_{1}\ol{w}_{j}-1)^{2}}}^{2}}&=O(N^{-2+\epsilon}), \\
			\frac{\kappa(k+1)}{k!}\sum_{a,b}\partial_{ab}^{k}[\Tr (G_{z_{1}}^{[21]}A_{1})\ctr{\caX^{(1)}\caY}]&\prec N^{-1/2-(k-4)_{+}/2}.
		\end{split}
		\eeq
		Note that the estimates in \eqref{eq:comput_real} are the same as \eqref{eq:comput1} and \eqref{eq:comput2} except for that of $\wt{\frX}_{i}$.
		Next, we explain how we modify the proof of \eqref{eq:comput1} to prove the first two estimates of \eqref{eq:comput_real}. The self-renormalization defined in \eqref{eq:defn_renorm} should be modified, where we took $\wt{W}$ to be the Hermitization of a complex Ginibre ensemble. When $X$ is real we take $\wt{W}$ to be that of a real Ginibre ensemble and define $\ul{Wf(W)}$ via the same formula. As a result, we have
		\beqs
			\ul{WG_{z}}	= WG_{z}-\E_{\wt{W}}\wt{W}(\partial_{\wt{W}}G_{z})=WG_{z}+\wt{\caS}[G_{z}]G_{z},
		\eeqs
		where $\wt{\caS}$ is defined in \eqref{eq:tildeS}. In this way, the analogue of \eqref{eq:renorm} becomes
		\beqs
			G_{z}-M_{z}=-M_{z}\ul{WG_{z}}+M_{z}\caS[G_{z}-M_{z}]G_{z}+M_{z}D_{z}G_{z},\qquad D_{z}\deq \frac{1}{N}
			\begin{pmatrix}
				0 & G_{\ol{z}}^{[12]} \\
				G_{\ol{z}}^{[21]} & 0
			\end{pmatrix},
		\eeqs
		where we used $(G_{z}^{[21]})^{\intercal}=G_{\ol{z}}^{[21]}$. Nonetheless, Lemma \ref{lem:2Gll} remains intact; following lines of \cite[Eq. (5.5)]{Cipolloni-Erdos-Schroder2019arXiv}, we have
		\beq
		\begin{split}
			\brkt{G_{z}PG_{w}}=&\brkt{M_{z}PM_{w}}+\brkt{M_{z}P(G_{w}-M_{w})}-\brkt{M_{z}\ul{WG_{z}PG_{w}}} +\brkt{M_{z}\wt{\caS}[G_{z}PG_{w}]M_{w}}	\\
			&+\brkt{M_{z}\wt{\caS}[G_{z}PG_{w}](G_{w}-M_{w})}+\brkt{M_{z}(\wt{\caS}[G_{z}]-\caS[M_{z}])G_{z}PG_{w}}	\\
			=&\brkt{M_{z}PM_{w}}+\brkt{M_{z}\caS[G_{z}PG_{w}]M_{w}}-\brkt{M_{z}\ul{WG_{z}PG_{w}}}+O_{\prec}(N^{-1}),
		\end{split}
		\eeq
		where we used the fact that
		\beqs
			\norm{(\caS-\wt{\caS})[G_{z}]}\leq \frac{1}{N}\norm{G_{z}}\prec \frac{1}{N},\qquad \norm{(\caS-\wt{\caS})[G_{z}PG_{w}]}\leq \frac{1}{N}\norm{G_{z}PG_{w}}\prec \frac{1}{N}.
		\eeqs
		Inspecting the proof of Proposition 5.3 in \cite{Cipolloni-Erdos-Schroder2019arXiv} which proves the bound $\brkt{\ul{WG_{z}PG_{w}}}\prec N^{-1}$, one can see that the same applies to the real case since $\ul{Wf(W)}$ still stands for the remainders of the (real) cumulant expansion. This leads to the same asymptotics for $\brkt{G_{z}PG_{w}}$ as in Lemma \ref{lem:2Gll}, that is, $\wh{\caB}^{-1}(M_{z}PM_{w})$, and similarly isotropic local law is valid. Therefore Lemma \ref{lem:2Gll} is true when $X$ is real as well.
		
		In the same spirit, using $\norm{(\wt{\caS}-\caS)[P]}\leq N^{-1}\norm{P}$ and $\norm{G}\prec 1$, all arguments along the proof of \eqref{eq:comput1} stay the same up to additional errors of $O_{\prec}(N^{-1})$. Thus the first two estimates of \eqref{eq:comput_real} reduce to
		\beq\label{eq:real_XY_result}
			\begin{split}
				&\Expct{\Absv{\brkt{G_{\ol{z}_{i}}^{[12]}G_{z_{1}}^{[21]}A_{1}G_{\ol{z}_{i}}^{[12]}A_{i}^{\intercal}}
					+\frac{\brkt{A_{1}A_{i}^{\intercal}}}{z_{i}(z_{1}z_{i}-1)}+\frac{z_{1}\brkt{A_{1}}\brkt{A_{i}^{\intercal}}}{(z_{1}z_{i}-1)^{2}}}^{2}}=O(N^{-2+\epsilon}), \\
				&\expct{\absv{\brkt{G_{\ol{w}_{j}}^{[22]}G_{z_{1}}^{[21]}A_{1}G_{\ol{w}_{j}}^{[11]}\ol{B}_{j}}}^{2}}=O(N^{-2+\epsilon}).
			\end{split}
		\eeq
		Since the left-hand sides of \eqref{eq:real_XY_result} have exactly the same form as $\frY_{j}$ and $\frX_{i}$, respectively, the result immediately follows.
		
		The proof of the last estimate in \eqref{eq:comput_real} is completely analogous to Section \ref{sec:prf_comput2}; we first exhaust the case $k\geq 4$ and use the exact same division of cases with $k_{i}+l_{i}$ replaced by $k_{i}$, where $k_{i}$ denotes the number of differential operators $\partial_{ab}$ hitting $(G_{1}^{[21]}A_{1})_{ba}$, $\caX_{i}$, or $\caY_{i-p}$. Then the proof immediately follows once we observe from 
		\beq\label{eq:real_der}
			\partial_{ab} G_{z}=G_{z}(\Delta_{ab}^{\{12\}}+\Delta_{ba}^{\{21\}})G_{z}
		\eeq
		that the real partial derivative has the same form as the sum of complex partial derivatives in the complex case;
		\beqs
			(\partial_{ab}^{(1,0)}+\partial_{ab}^{(0,1)})G_{z}=G_{z}(\Delta_{ab}^{\{12\}}+\Delta_{ba}^{\{12\}})G_{z}.
		\eeqs 
		Therefore, after expanding all real partial derivatives using \eqref{eq:real_der}, each individual term has already been covered in Section \ref{sec:prf_comput2}. This completes the proof of \eqref{eq:comput_real}.
		
		Finally, to prove \eqref{eq:mean_asymp_G} for real-valued $X$, we first rewrite \eqref{eq:mean_expa};
		\beq\label{eq:mean_G_real}
			\expct{\ii\eta\Tr G^{[11]}_{1}A_{1}}=-z_{1}\expct{\Tr G^{[21]}_{1}A_{1}}-\Tr A_{1}
			+\sum_{k\in\N}\frac{\kappa(k+1)}{k!}\sum_{a,b}\expct{\partial_{ab}^{k}[\Tr \Delta_{ab}G^{[21]}_{1}A_{1}]}.
		\eeq
		The only difference here is in the first term of the cumulant expansion. Taking $k=1$, we get the same quantity as in \eqref{eq:real_Stein_1},
		\beqs
			-\frac{1}{N}\sum_{a,b}\expct{\Tr G_{1}(\Delta_{ab}^{\{12\}}+\Delta_{ba}^{\{21\}})G_{1}(A_{1}\Delta_{ab})^{\{12\}}}
			=-\expct{\brkt{(G_{z_{1}}A_{1}^{\{11\}}G_{\ol{z}_{1}})^{[22]}}}-\expct{\brkt{G_{z_{1}}^{[22]}\Tr G_{z_{1}}^{[11]}A_{1}}},
		\eeqs
		so that the first term is newly appeared in the real case. By Lemma \ref{lem:2Gll} and \eqref{eq:2G_mat}, we have
		\beqs
			\expct{\brkt{(G_{z_{1}}A_{1}^{\{11\}}G_{\ol{z}_{1}})^{[22]}}}=\brkt{M_{A_{1}^{\{11\}}}(z_{1},\ol{z}_{1})^{[22]}}+O(N^{-1+\epsilon})=\frac{\brkt{A_{1}}}{z_{1}^{2}-1}.
		\eeqs
		Plugging this into \eqref{eq:mean_G_real} proves \eqref{eq:mean_asymp_G} for the real case. This concludes the proof of Proposition \ref{prop:CLT_G} for real-valued $X$.

\end{document}